\documentclass[11pt]{article}
\usepackage [dvips]{graphics}
\usepackage[centertags]{amsmath}
\usepackage{amsfonts}
\usepackage{amssymb}
\usepackage{amsthm}
\usepackage{newlfont}
\usepackage{stmaryrd}
\usepackage{color,epsfig}
\usepackage{amsfonts,graphicx,psfrag}
\usepackage{graphicx}
\usepackage{float}

\usepackage{txfonts}
\usepackage{mathrsfs}

\theoremstyle{plain}
\newtheorem{thm}{Theorem}[section]
\newtheorem{cor}[thm]{Corollary}
\newtheorem{lem}[thm]{Lemma}
\newtheorem{prop}[thm]{Proposition}

\newtheorem{rem}[thm]{Remark}

\def\sqr#1#2{{\vcenter{\vbox{\hrule height.#2pt
              \hbox{\vrule width.#2pt height#1pt \kern#1pt \vrule
width.#2pt}
              \hrule height.#2pt}}}}
\def\be{\begin{equation}}
\def\ee{\end{equation}}

\def\ga{{\gamma}}

\def\om{{\omega}}

\def\ep{{\epsilon}}

\def\Sp{{\mathrm {Sp}}}
\def\lb{\label}

\def\ga{{\gamma}}

\def\R{{\Bbb R}}

\def\U{{\Bbb U}}

\def\<{{\langle}}
\def\>{{\rangle}}
\def\no{\noindent}

\def\bs{\bigskip}

\def\dim{\hbox{\rm dim$\,$}}

\def\({\Big (}
\def\){\Big )}
\def\[{\Big[}
\def\]{\Big]}

\def\be{\begin{equation}}
\def\bel{\begin{equation}\label}
\def\ee{\end{equation}}
\def\bea{\begin{eqnarray}}
\def\eea{\end{eqnarray}}
\def\bt{\begin{theorem}}
\def\et{\end{theorem}}
\def\bc{\begin{corollary}}
\def\ec{\end{corollary}}
\def\bl{\begin{lemma}}
\def\el{\end{lemma}}
\def\bp{\begin{proposition}}
\def\ep{\end{proposition}}
\def\br{\begin{remark}}
\def\er{\end{remark}}
\def\ba{\begin{array}}
\def\ea{\end{array}}
\def\bd{\begin{definition}}
\def\ed{\end{definition}}

\makeatletter
   
   \@addtoreset{equation}{section}
\makeatother

\begin{document}

\title{\bf  Linear stability of the elliptic  relative equilibrium with  $(1+n)$-gon central configurations in planar $n$-body problem }
\author{Xijun Hu$^{a, 1}$\thanks{Partially supported
by NSFC (Nos. 11790271, 11425105), E-mail:xjhu@sdu.edu.cn.}
\quad Yiming Long$^{b, c, 2}$ \thanks{Partially supported
by NSFC (Nos.11131004, 11671215 and 11790271), LPMC of Ministry of Education of China, Nankai University, Nankai Zhide Foundation, and the Beijing Advanced Innovation Center for Imaging Technology at Capital Normal University, E-mail:longym@nankai.edu.cn.}
\quad Yuwei Ou$^{d, 3}$ \thanks{Partially supported
by NSFC (Nos.11801583,11671215), E-mail:ouyw3@mail.sysu.edu.cn.}
\\ \\
$^{a}$ Department of Mathematics, Shandong University
Jinan, Shandong 250100\\
$^{b}$ Chern Institute of Mathematics and LPMC, Nankai University
Tianjin 300071\\
$^{c}$ Beijing Advanced Innovation Center for Imaging Technology,\\ Capital Normal University, Beijing 100048\\
$^{d}$ School of Mathematics(Zhuhai),  Sun Yat-Sen University, Zhuhai, Guangdong
519082\\
The People's Republic of China
}
\date{}
\maketitle
\begin{abstract} We study the linear stability of  $(1+n)$-gon elliptic  relative equilibrium (ERE for short),
that is the Kepler homographic solution with the $(1+n)$-gon central configurations. We show that for $n\geq 8$
and any eccentricity $e\in[0,1)$, the $(1+n)$-gon ERE is stable when the central mass $m$ is large enough. Some
linear instability results are given when $m$ is small.
\end{abstract}

\bs

\no{\bf AMS Subject Classification:} 37J25,  70F10,   37J45, 53D12

\bs

\no{\bf Key Words:}  linear stability, elliptic relative equilibrium, Maslov index, planar $n$-body problem

\section{Introduction}

For $n$ particles with masses $m_1,\cdots,m_n$, let $q_1,\cdots,q_n\in \mathbb{R}^2$ be the position vectors. Let
\bea U=\sum_{1\leq i< j\leq n} \frac{m_im_j}{\|q_i-q_j\|} \eea
be the negative potential function defined on the configuration space  $$
\Lambda=\{x=(x_1,\cdots,x_n)\in\mathbb{R}^{2n}\setminus\triangle:
\sum_{i=1}^nm_ix_i=0 \},$$ where
$\triangle=\{x\in\mathbb{R}^{2n}:\exists i\neq j,x_i=x_j \}$ is the
collision set.  Obviously, the orbits of the $n$ bodies satisfy the following Newton equation
\bea m_i\ddot{q}_i(t)=\frac{\partial
U}{\partial q_i}(q_1,...,q_n). \label{1.2} \eea

An elliptic relative equilibrium is a special solution of the planar $n$-body problem, which is generated by a central
configuration. A central configuration is formed by $n$ position vectors $(q_1,...,q_n)=(a_1,...,a_n)$ which satisfy
\bea -\lambda m_jq_j=\frac{\partial U}{\partial q_j}(q_1,...,q_n) \eea
for some constant $\lambda$.  An easy computation shows that $\lambda=U(a)/\cal{I}(a)>0$, where
$\cal{I}(a)=\sum m_j\|a_j\|^2$ is the moment of inertia. In other words, a central configuration with
$\cal{I}(a)=1$ is a critical point of the function $U$ restricted to the set
$\mathcal{E}=\{x\in\Lambda \mid \cal{I}(x)=1\}$.

A planar central configuration of the $n$-body problem gives rise to a solution of (\ref{1.2}) where each particle moves on a
specific Keplerian orbit while the totality of the particles move on a homothety motion. If the Keplerian orbit
is elliptic then the solution is an equilibrium in pulsating coordinates so we call this solution an elliptic
relative equilibrium (ERE for short), and a relative equilibrium in case $e=0$ (cf. \cite{MS}).

From Meyer-Schmidt \cite{MS}, there are two four-dimensional invariant symplectic subspaces, $E_1$ and $E_2$,
and they are associated to the translation symmetry, dilation and rotation symmetry of the system. In other words,
there is a symplectic coordinate system in which the linearized system of the planar $n$-body problem decouples
into three subsystems on $E_1, E_2$ and $E_3 = (E_1 \cup E_2)^\perp$, where $\perp$ denotes the symplectic
orthogonal complement. A symplectic matrix $M$ is called spectrally stable if all eigenvalues of $M$ belong to the
unit circle $\mathbb{U}$ of the complex plane. $M$ is called linearly stable if it is spectrally stable and semi-simple.
While $M$ is called hyperbolic if no eigenvalues of $M$ are on $\mathbb{U}$. The ERE is called hyperbolic
(stable, resp.) if the monodromy matrix $M$ restricted to $E_3$, $M|E_3$,
is hyperbolic (stable, resp.).

More precisely, Let $I_j$ be the identity matrix on $\mathbb{R}^j$ and
$J_{2j}=\left( \begin{array}{cccc}0_j& -I_j \\
                                  I_j& 0_j \end{array}\right)$. Here
we always omit the subscript of $J$ when there is no confusion.  Let ($\R^{2n}, \omega$) with $\omega(x,y)=(Jx,y)$
be the standard symplectic space, and we denote by
$$   \Sp(2n)=\{M\in GL(2n),  M^TJM=J\}  $$
the symplectic group. As in \cite{Lon4}, for $M_1=\left( \begin{array}{cccc}A_1& A_2 \\
A_3 & A_4 \end{array}\right)$,  $M_2=\left( \begin{array}{cccc}B_1& B_2 \\
B_3 & B_4 \end{array}\right)$, the symplectic sum $\diamond$ is defined by
\bea  M_1\diamond M_2=\left( \begin{array}{cccccccc}A_1& 0 &A_2 & 0 \\
                                                     0 & B_1& 0  & B_2 \\
                                                    A_3& 0 &A_4 & 0 \\
                                                     0 & B_3& 0  & B_4 \end{array}\right).   \eea
For $M_1, M_2\in \Sp(2n)$, we denote by $M_1\approx M_2$ if there exists a $P\in\Sp(2n)$, such that $M_1=P^{-1}M_2P$ holds.
We set  ${\mathcal R}(\theta_{1})=\left( \begin{array}{cc} \cos \theta_{1} & -\sin \theta_{1}\\
\sin \theta_{1} & \cos \theta_{1}\end{array}\right)\in\Sp(2)$. Then $M\in\Sp(2n)$ is linearly stable if and only if $$M\approx {\mathcal R}(\theta_1)\diamond\cdots\diamond {\mathcal R}(\theta_n), \quad \theta_i\in [0,2\pi).$$
Meyer-Schmidt's result shows that for $T>0$ a $T$-periodic ERE satisfies the linear system
\bea \dot{\xi}=JB \xi,  \eea
with $B=B_1\diamond B_2\diamond B_3$, where $B_1$ is associated to the translation symmetry, $B_2$ is associated to
the dilation and rotation symmetries of the system which is just the linear part of the Kepler orbits, $B_3$ is
the essential part. Let $\ga$ be the fundamental solution of $B_3$, that is $$\dot{\ga}=JB_3\ga,$$ then $\ga(T)$ is
just the monodromy matrix $M$ restricted to $E_3$.

For $n=3$, there are only two kinds of central configurations, the Lagrangian equilateral triangle central
configuration and Euler collinear central configurations.   There are many works on the linear stability of the elliptic Lagrangian orbits and elliptic Euler orbits, please refer to \cite{HLS},  \cite{MSS1}, \cite{MSS2} , \cite{ZL} and reference therein.
For $n\geq 4$, it is difficult to find all central
configurations. It is easy to see that the $(1+n)$-gon central configuration exists for any $n\in\mathbb N$,
where $n$ equal masses $m_{k}$ are at the vertices of a regular $n$-gon with an additional mass $m$ at the center.
Without loss of generality, we set $m_{k}=1$, for $k\in\{1,...,n\}$ and let $m$ represent the mass of the
body at the center. It is natural to treat $m$ as a parameter.

There have existed many works which studied the linear stability of relative equilibria of the $(1+n)$-gon,
i.e., the case with $e=0$. As far as we know, this was first started by Maxwell in his study on the stability of
Saturn's rings (cf. \cite{Max1,Max2}). Moeckel \cite{Moe1} proved that the $(1+n)$-gon is linearly stable for sufficiently
large $m$ only when $n\geq7$. For $n\geq 7$, Roberts found a value $h_n$ which is proportional to $n^3$, and the
$(1+n)$-gon is stable if and only if $m>h_n$ (cf. \cite{Rob1}). For other related works, please refer to \cite{VK} and reference therein.

A question proposed by Moeckel is that for a linearly stable relative equilibrium ($e=0$), is there always a dominant
mass, i.e., a body with a mass which is much larger than the total mass of the other bodies? Another question is
whether the linearly stable relative equilibrium is always a non-degenerate minimum of the $U|_\mathcal{E}$
(cf. \cite{ACS}, Problem 15, 16).

Moeckel's conjecture is true for the relative equilibrium of $(1+n)$-gon, but we are not aware of such a result for
elliptic relative equilibrium. In this paper, we study the linear stability of  $(1+n)$-gon ERE. Our next main Theorems
1.1 and 1.3 show that Moeckel's conjecture is also true when $e>0$, specially Moeckel's conjecture holds for
$(1+n)$-gon EREs when $n\ge 8$.

Since the $(1+n)$-gon possesses a rotational symmetry, the linear system of its essential part can be decomposed into
$[n/2]$ linear sub-systems. By change of variables (cf. \cite{MS}), we can suppose that the linear system of the
essential part of the ERE of the $(1+n)$-gon is given by
\bea \frac{d\gamma}{d\theta}=J_{4(n-1)}B_3(e,\theta) \gamma, \eea
where $e$ is the eccentricity and $\theta\in[0,2\pi]$ is the true anomaly. Then
\bea
B_{3}(e,\theta)=\mathcal{B}_1(e,\theta) \diamond\cdots\diamond \mathcal{B}_{[\frac{n}{2}]}(e,\theta). \lb{bdcom.0} \eea
Let $\ga_l$ be the fundamental solution of $\mathcal{B}_l(e,\theta)$ for $l=1,\cdots,[n/2]$, then
\bea  \ga=\ga_1\diamond\cdots\diamond\ga_{[\frac{n}{2}]}. \eea
Please refer to Theorem \ref{thm3.4} below for the details.

Obviously, $\ga$ is stable if and only if each $\ga_l$ is stable for $l=1,\cdots, [n/2]$.

\begin{thm}\label{th1.1}
For $n\geq4$ and any $e\in [0,1)$, each $\ga_l$ with $l=2,\cdots,[n/2]$ is linearly stable when $m$ is large enough and they have
the normal form below:

i)\ \ For $2\leq l\leq[\frac{n-1}{2}]$, $\ga_{l}(2\pi)\approx {\mathcal R}(\alpha_{l})\diamond {\mathcal R}(\beta_{l})
\diamond {\mathcal R}(\theta_{l})\diamond {\mathcal R}(\phi_{l})$ for some $\alpha_{l}, \beta_{l}, \theta_{l}$ and
$\phi_{l}\in(\pi, 2\pi)$.

ii)\ \ For $n\in2\mathbb{N}, l=[\frac{n}{2}]$, $\ga_{l}(2\pi)\approx {\mathcal R}(\alpha_{l})\diamond {\mathcal R}(\beta_{l})$
for some $\alpha_{l}$ and $\beta_{l}\in(\pi, 2\pi)$.
\\
Moreover, for $n\geq 8$ and $e\in [0,1)$, $\ga_1$ is linearly stable when $m$ is large enough, and $\ga_1(2\pi)
\approx {\mathcal R}(\alpha_{1})\diamond {\mathcal R}(\beta_{1})
\diamond {\mathcal R}(\theta_{1})\diamond {\mathcal R}(\phi_{1})$ for some $\alpha_{1}, \beta_{1}, \theta_{1}$ and
$\phi_{1}\in(\pi, 2\pi)$.
Consequently the $(1+n)$-gon ERE is
stable in this case.


\end{thm}

\begin{rem} For $n=2,\cdots,6$, $\ga_1$ is not linearly stable even in the case $e=0$. $n=2$ is a special case of elliptic
Euler orbits and was studied in \cite{HO16} and \cite{ZL}. For $n=7$,  $\ga_1$ is stable when $e=0$ and $m$ large enough. We guess that this
is also true for any $e\in (0,1)$. It is not clear to us whether the method in this paper can be used to solve it.
\end{rem}

The idea of the proof of the above theorem is based on the analysis of corresponding Sturm-Liouville operators and
the Maslov-type index theory (cf. \cite{Lon4}). For reader's convenience, instead of introducing the Maslov-type index
theory, we give the stability criteria in terms of the Morse indices. Our method can also be used to study the
hyperbolicity when $m$ is small.

\begin{thm}\label{th1.2}
Let $n\geq 3$. For $l=1,\cdots,[n/2]$, if $m\in (\Gamma^-_l, \Gamma^+_l)$, then $\ga_l$ is hyperbolic for all $e\in [0,1)$,
where $\Gamma^\pm_l$ are given in Theorem \ref{thm5.1} below.
\end{thm}

Consequently, we have much stronger results for $n=3$, $4$ and $5$:

$(1+3)$-gon system is hyperbolic for all $m\in [0, 0.0722)$ and $e\in [0,1)$;

$(1+4)$-gon system is hyperbolic for all $m\in [0, 0.1768)$ and $e\in [0,1)$;

$(1+5)$-gon system is hyperbolic for all $m\in (0.2613, 0.3035)$ and $e\in [0,1)$.

In fact, we guess that $(1+5)$-gon system is hyperbolic for all $m\in [0, 0.3035)$ and $e\in [0,1)$.

This paper is organized as follows. In Section 2, we explain the reduction results of $(1+n)$-gon ERE.
We introduce criteria for related operators and study their properties in Section 3. We prove the
stability Theorem \ref{th1.1} in Section 4, and then we study the unstable cases and prove Theorem \ref{th1.2}
in Section 5.

\section{The Reduction of Elliptic Relative Equilibria of $(1+n)$-gon.}

In 2005, Meyer and Schmidt used the central configuration coordinate to reduce the elliptic relative equilibria
and get the essential part for the linear stability. Their central configuration coordinate is very important for us to reduce
the $(1+n)$-gon ERE.  For the reader's convenience, we briefly review the central configuration
coordinates introduced by Meyer and Schmidt in \cite{MS}.

Considering $n$ particles with masses $m_1,...,m_n$, let $Q=(q_1,...,q_n)\in (\R^2)^n$ be
the position vector, and $P=(p_1,...,p_n)\in (\R^2)^n$ be the momentum vector. Denote by $d_{ij}=||q_{i}-q_{j}||$.
The Hamiltonian function has the form
\bea
H(P,Q)=\sum_{j=1}^n\frac{||p_{j}||^2}{2m_{j}}-U(Q),\ \ U(Q)=\sum_{1\leq j<i\leq n}\frac{m_{j}m_{i}}{d_{ji}}.
\eea
We denote by $\mathbb{J}_n=diag(J_2,...,J_2)_{2n\times2n} $ and  $M=diag (m_1,m_1,m_2,m_2,...,m_n, m_n)_{2n\times2n}$.
Let $x(t)$ be a periodic ERE solution with respect to a central configuration $a$.  Then the corresponding
fundamental solution $\ga$ is given by
\bea \dot{\ga}(t)=J_{4n}H''(x(t))\ga(t),\,\ \ga(0)=I_{4n}. \label{ga}  \eea
As in \cite{MS} (page 266, Cor. 2.1), for the homographic solution $(P(t),Q(t))$ of a central configuration $a$,
by using the central configuration coordinate, the system (\ref{ga}) can be decomposed into $3$ subsystems on
$E_1$, $E_2$ and $E_3=(E_1\cup E_2)^{\perp}$ respectively. A basis of $E_1$ is given by $(u,0)$, $(v,0)$, $(0,Mu)$,
and $(0,Mv)$, where $u =(1, 0, 1, 0, ...)$, $v =(0,1, 0, 1, .)$. The space $E_2$ is spanned by $(a,0)$,
$(\mathbb{J}_na,0)$, $(0,Ma)$, and $(0,\mathbb{J}_nMa)$. $E_{1}$ reflects the translation invariant of the problem;
$E_{2}$ is the space swept out by rotation and dilation of central configurations; and $E_{3}$ is the essential part.

Meyer and Schmidt first introduced the linear transformation of the form $Q = AX,  P = A^{-T} Y$ with
$X = (g, z,w)\in \mathbb{R}^2\times\mathbb{R}^2\times\mathbb{R}^{2n-4}$ and
$Y =(G,Z,W)\in\mathbb{R}^2\times\mathbb{R}^2\times\mathbb{R}^{2n-4}$, where
$A\in GL(\R^{2n})$ and satisfies (cf. \cite{MS}, p.263)
\bea  \mathbb{J}_{n}A = A\mathbb{J}_{n}, \qquad A^TMA = I_{2n}.  \label{AA}\eea
After this transformation, $B(t)=H''(x(t))$ in this new coordinate system  has the form
$B(t)=B_1(t)\oplus B_2(t)\oplus B_3(t)$, where $B_i(t)=B|_{E_i}(t)$.
The essential part $B_3(t)$ is a path of $(4n-8)\times(4n-8)$ symmetric matrices.

By taking the rotating coordinates and using the true anomaly $\theta$ as the variables, Meyer and Schmidt \cite{MS}
gave a useful form of the essential part, that is
\bea B_{3}(\theta)=\left( \begin{array}{cccc} I_{k} & -\mathbb{J}_{k/2} \\
\mathbb{J}_{k/2} & I_{k}-r_e(\theta)(I_{k}+\mathcal {D})
\end{array}\right),\quad \theta\in[0,2\pi], \label{msf} \eea
where $k=2n-4$ and $e$ is the eccentricity, $r_e(\theta)=(1+e\cos(\theta))^{-1}$  and
\bea \mathcal{D}=\frac{1}{\lambda}A^TD^2U(a)A\big|_{w\in\mathbb{R}^{k}},
      \quad {\mathrm with} \quad \lambda=\frac{U(a)}{I(a)}.  \eea

We denote by $R:= I_k + \mathcal{D}$, which can be considered as the regularized Hessian of
the central configurations. In fact, direct computations show that
\bea  R=\frac{1}{U(a)}A^TD^2U(a)|_{\Sigma}A\big|_{w\in\mathbb{R}^{k}},   \eea
and the corresponding Sturm-Liouville system is
\bea -\ddot{y}-2\mathbb{J}_{k/2}\dot{y}+r_e(\theta) R y=0.  \label{st} \eea
Let $\ga_e(\theta)$ be the fundamental solution of $B_3$, that is
\bea \dot{\ga}_e(\theta)=JB_3(\theta)\ga_e(\theta),  \quad  \ga_e(0)=I_{2n}. \label{ga3}
\eea
The ERE is spectrally stable (hyperbolic), if $\ga_e(2\pi)$ is spectrally stable (hyperbolic).
Let $a=(x^T_{0},x^T_{1},...x^T_{n})^T$ be the position vector of the $(1+n)$-gon central
configuration with $x_{0}=(0,0)^T$, $x_{k}=(\cos\theta_{k},\sin\theta_{k})^T$, where
$\theta_{k}=\frac{2\pi k}{n}, k\in\{1,2,...n\}$ and $M=diag(m, m,1,1,...,1,1)$.

In order to get the exact form of $A^TU''(a)A$, the first step is to find a series of invariant subspaces
$W_{l}$ with $l\geq 1$ of $M^{-1}U''(a)$, the second step is to find the $M$-orthogonal bases of $W_{l}$.
Here, two vector $u\neq v$ are called $M$-orthogonal if $u^TMv=0$ and $u^TMu=1$ hold. Then all the $M$
orthogonal bases form the matrix $A$, also we can get a series of exact expressions of $M^{-1}U''(a)$
corresponding to each invariant subspaces $W_{l}$.

The construction of the invariant subspace $W_{l}$ was given in \cite{Max1, Moe1} in the study of the
case of $e=0$. In fact, they can be obtained as follows.

Let $\omega=e^{\frac{2\pi i}{n}}$, $\omega^{kl}=e^{\frac{2\pi i}{n}kl}$ for $1\leq k\leq n$ and $l\geq 0$,
\bea
S=\left( \begin{array}{ccccc} O_{2} & I_{2} & O_{2}& ...& O_{2} \\
O_{2} & O_{2} & I_{2} & ... & O_{2}\\
\vdots & \vdots & \vdots & ... & \vdots\\
O_{2} & O_{2} & O_{2} &... & I_{2}\\
I_{2} & O_{2} & O_{2} &... & O_{2}
\end{array}\right)_{2n\times 2n},\eea
and $R_{n}(\theta_{1}) = diag({\mathcal R}(\theta_{1}),{\mathcal R}(\theta_{1}),...{\mathcal R}(\theta_{1}))_{2n\times 2n},\,
\hat{S}=diag({\mathcal R}(\theta_{1})^{-1},{\mathcal R}_{n}(\theta_{1})^{-1}S)$. Since $\hat{S}(a)=a$, we have the lemma
below which is got by direct computations.

\begin{lem}
We have $U(\hat{S}y)=U(y)$ for every $y\in(\R^2)^{1+n}$. Here especially for every $(1+n)$-gon central
configuration $a$, the identity $\hat{S}U''(a)=U''(a)\hat{S}$ holds. Consequently from the fact
$M^{-1}\hat{S}=\hat{S}M^{-1}$, the identity $\hat{S}M^{-1}U''(a)=M^{-1}U''(a)\hat{S}$ holds. Hence
each eigen-subspace of $\hat{S}$ must be an invariant subspace of $M^{-1}U''(a)$. \label{lem3.1}
\end{lem}

Based on Lemma \ref{lem3.1}, it suffices to find all the eigen-subspaces of $\hat{S}$. Then we choose
the $M$-orthogonal bases of each one of these subspaces and compute the reduction form of
$M^{-1}U''(a)$. The results below are taken from Moeckel \cite{Moe1} and Roberts \cite{Rob1}.

\begin{lem}
The following subspaces are the invariant subspaces of $M^{-1}U''(a),$
\bea
W(0)&=&\mathrm{Span}\{a, \mathbb{J}_{n+1}a\},\nonumber \\
W(1)&=&\mathrm{Span}\{\hat{c}, \mathbb{J}_{n+1}\hat{c}, \hat{v}(1),\mathbb{J}_{n+1}\hat{v}(1),
         \hat{w}(1), \mathbb{J}_{n+1}\hat{w}(1)\},\nonumber \\
W(l)&=&\mathrm{Span}\{v(l),\mathbb{J}_{n+1}v(l),w(l),\mathbb{J}_{n+1}w(l)\},\quad 2\leq l\leq[\frac{n-1}{2}],\nonumber \\
W(\frac{n}{2})&=&\mathrm{Span}\{v(\frac{n}{2}),\mathbb{J}_{n+1}v(\frac{n}{2})\}, \quad\mathrm{if}\ \ n\in 2\mathbb{N}.\nonumber
\eea
where
\bea
\hat{w}(1)&=&(0,0,\cos2\theta_{1},\sin2\theta_{1},\cdots,\cos2\theta_{n},\sin2\theta_{n})^T,\nonumber \\
\hat{c}&=&(1,0,1,0,\cdots,1,0)^T,\ \ \hat{v}(1)=(-n/m,0,1,0,\cdots,1,0)^T,\nonumber \\
v(l)&=&(0, 0, v_{1l}, \cdots , v_{nl})^T,w(l)=(0,0,w_{1l},\cdots,w_{nl})^T, \nonumber \\
v_{kl}&=&\cos \theta_{kl}\cdot(\cos \theta_{k},\sin \theta_{k}),\ \
w_{kl}=\sin \theta_{kl}\cdot(\cos \theta_{k},\sin \theta_{k}).\nonumber
\eea
\end{lem}
Direct computations show that
\bea
a^TMa&=&n,\ \ \hat{c}^TM\hat{c}=m+n,\ \ \hat{v}(1)^TM\hat{v}(1)=\frac{n^2}{m}+n, \nonumber \\
\hat{w}(1)^TM\hat{w}(1)&=&n,\ \ v(l)^TMv(l)=\frac{n}{2},\ \ w(l)^TMw(l)=\frac{n}{2}. \nonumber
\eea
Then we normalize the bases as follows,
\bea
\frac{a}{\sqrt{n}},\ \ \frac{\hat{c}}{\sqrt{m+n}},\ \ \sqrt{\frac{m}{n^2+m n}}\hat{v}(1),\ \
\frac{1}{\sqrt{n}}\hat{w}(1),\ \ \sqrt{\frac{2}{n}}v(l),\ \ \sqrt{\frac{2}{n}}w(l).\nonumber
\eea
After the normalization, all the bases are $M$-orthogonal.

Now we construct the matrix $A$ by using the bases of $W(l)$. Let
\bea
A=\left( \begin{array}{cccc}
A_{11} & A_{12} & \cdots & A_{1(n+1)} \\
A_{21} & A_{22} & \cdots & A_{2(n+1)}\\
\cdots & \cdots & \cdots & \cdots\\
A_{(n+1)1} & A_{(n+1)2} & \cdots & A_{(n+1)(n+1)}\end{array}\right),\nonumber
\eea
where each $A_{ij}$ is defined by
\bea
A_{cen}=\left( \begin{array}{c}A_{11} \\ A_{21} \\\vdots \\A_{(n+1)1}\end{array}\right)
=\frac{1}{\sqrt{m+n}}(\hat{c}, \mathbb{J}_{n+1}\hat{c}), \ \
A_{kep}=\left( \begin{array}{c}A_{12} \\ A_{22} \\\vdots \\A_{(n+1)2}\end{array}\right)
=\frac{1}{\sqrt{n}}(a,\mathbb{J}_{n+1}a),\nonumber
\eea
\bea
A(1)=\left( \begin{array}{cc}A_{13} & A_{14} \\ A_{23} & A_{24} \\ \vdots & \vdots \\A_{(n+1)3} & A_{(n+1)4}\end{array}\right)
&=&(\sqrt{\frac{m}{n^2+m n}}\hat{v}(1),\sqrt{\frac{m}{n^2+m n}}\mathbb{J}_{n+1}\hat{v}(1),
\frac{1}{\sqrt{n}}\hat{w}(1),\frac{1}{\sqrt{n}}\mathbb{J}_{n+1}\hat{w}(1)),\nonumber \\
A(l)=\left( \begin{array}{cc}A_{1(2l+1)} & A_{1(2l+2)} \\ A_{2(2l+1)} & A_{2(2l+2)} \\ \vdots & \vdots \\A_{(n+1)(2l+1)} &
A_{(n+1)(2l+2)}\end{array}\right)
&=&\sqrt{\frac{2}{n}}(v(l),\mathbb{J}_{n+1}v(l),w(l),\mathbb{J}_{n+1}w(l)),\ \
2\leq l\leq [\frac{n-1}{2}]\ \nonumber \\
A(\frac{n}{2})=\left( \begin{array}{c}A_{1(2l+1)} \\ A_{2(2l+1)} \\ \vdots  \\A_{(n+1)(2l+1)} \end{array}\right)
&=&\sqrt{\frac{2}{n}}(v(l),\mathbb{J}_{n+1}v(l)),\ \ l=[\frac{n}{2}],\ \ \mathrm{if}\ \ n\in 2\mathbb{N}. \label{3.24}\nonumber
\eea
Then the matrix $A$ satisfies $A^TMA=I_{2(n+1)}$ and $A\mathbb{J}_{n+1}=\mathbb{J}_{n+1}A$ as required in (\ref{AA}).

In order to get the essential part of the Hessian, it suffices to compute $\mathcal{U}=A^{T}U''(a)A|_{\mathcal{E}}$.
By (\ref{AA}) we have $M^{-1}U''(a)A=A\mathcal{U}$. By using the properties of the matrix $A$, we can define
$$  \mathcal{U}=diag(\mathcal{U}_{cen},\mathcal{U}_{kep},\mathcal{U}(1),...,\mathcal{U}(l),...,\mathcal{U}(\frac{n}{2})), $$
where they satisfy
\bea
M^{-1}U''(a)A_{cen}=A_{cen}\mathcal{U}_{cen},\ \
M^{-1}U''(a)A_{kep}=A_{kep}\mathcal{U}_{kep},\ \
M^{-1}U''(a)A(l)=A(l)\mathcal{U}(l). \nonumber
\eea
Then $\mathcal{U}$ can be decomposed into a series of parts $\mathcal{U}(l)$, where $\mathcal{U}_{cen}$
corresponds to the motion of the center of mass, $\mathcal{U}_{kep}$ corresponds to the Kepler problem and
the rest parts $\mathcal{U}(l)$ with $1\leq l\leq [\frac{n}{2}]$ correspond to the essential parts which
describe the linear stability of the homographic solution of the $(1+n)$-gon problem. We will get the
precise form of $\mathcal{U}(l), 1\leq l\leq [\frac{n}{2}]$ below.

Let
\bea
a_{0}=\sigma_{n}+2m,\ \
b_{0}=-\frac{1}{2}\sigma_{n}-m,\quad with \quad  \sigma_{n}=\frac{1}{2}\sum_{i=1}^{n-1}\csc \frac{\pi i}{n},\nonumber
\eea
and
\bea
a_{l}=P_{l}-3Q_{l}+2m,\ \ b_{l}=P_{l}+3Q_{l}-m,\nonumber
\eea
\bea
P_{l}=\sum_{j=1}^{n-1}\frac{1-\cos \theta_{jl}\cos \theta_{j}}{2d_{nj}^3},\ \
S_{l}=\sum_{j=1}^{n-1}\frac{\sin \theta_{jl}\sin \theta_{j}}{2d_{nj}^3},\ \
Q_{l}=\sum_{j=1}^{n-1}\frac{\cos \theta_{j}-\cos \theta_{jl}}{2d_{nj}^3}.
\eea
Now we write all parts $\mathcal{U}(l)$ of $M^{-1}U''(a)$ in the
new coordinate, and list them below.
\bea
\mathcal{U}(0)=\left(\begin{array}{cc}a_{0} & 0\\
0 & b_{0}
\end{array}\right),
\eea
\bea\label{2.17}
\mathcal{U}(1)=
\left(\begin{array}{cccc}\frac{n+m}{2}& 0 &\frac{3}{2}\sqrt{m(m+n)}& 0\\
0 & \frac{n+m}{2}& 0 & -\frac{3}{2}\sqrt{m(m+n)}\\ \frac{3}{2}\sqrt{m(m+n)} & 0 & \frac{m}{2}+2P_{1} & 0\\
0 & -\frac{3}{2}\sqrt{m(m+n)} & 0 & \frac{m}{2}+2P_{1}\end{array}\right),
\eea
\bea
\mathcal{U}(l)=\left(\begin{array}{cccc}a_{l} & 0 & 0 & S_{l}\\
0 & b_{l} & -S_{l} & 0\\ 0 & -S_{l} & a_{l} & 0\\
S_{l} & 0 & 0 & b_{l}\end{array}\right),\ \ \begin{array}{c} 2\leq l\leq [\frac{n-1}{2}]
\end{array},\nonumber
\eea
\bea
\mathcal{U}(\frac{n}{2})=\left(\begin{array}{cc}
P_{\frac{n}{2}}-3Q_{\frac{n}{2}}+2m & 0 \\ 0 & P_{\frac{n}{2}}+3Q_{\frac{n}{2}}-m\end{array}\right),\ \
       \mathrm{if}\ \ n\in 2\mathbb{N}.
\eea
Direct computations show that $\lambda=\frac{1}{2}\sigma_{n}+m$. Let
\bea
\mathcal{B}_{Kep}(\theta)=\left( \begin{array}{cccc} I & -\mathbb{J}\\
\mathbb{J} & I-r_e(\theta)R_{Kep}  \end{array}\right), \quad where \quad R_{Kep}=I+\frac{1}{\lambda}\mathcal{U}_{Kep},\nonumber
\eea
\bea\label{2.13}
\mathcal{B}_{l}(\theta)=\left( \begin{array}{cccc} I & -\mathbb{J}\\
\mathbb{J} & I-r_e(\theta)R_l  \end{array}\right), \quad where \quad R_l=I+\frac{1}{\lambda}\mathcal{U}(l), \, l=1,\cdots,[\frac{n}{2}].  \eea
Here we omit the sub-indices of $I$ and $\mathbb{J}$, which are chosen to have the same dimensions as those of $R_l$.

Now we get the theorem below,

\begin{thm}\label{thm3.4}
In the new coordinates, by restricting to the configuration space $\Lambda$, the linear Hamiltonian system
for the elliptic  $(1+n)$-gon homographic solution
$\bar{\xi}_{0}=(Y_{0}(\theta),X_{0}(\theta))^T=(0,\sigma,0,...,0,\sigma,0,...,0)^T\in\R^{4n}$
is given by
\bea  \dot{\zeta}(\theta)=J_{4n}B(\theta)\zeta(\theta),\label{2.14} \eea
with $B(\theta)=B_{2}(\theta)\diamond B_{3}(\theta)$, where $B_{2}(\theta)=\mathcal{B}_{Kep}(\theta)$ corresponds to the
linearized system of the Kepler 2-body problem at the Kepler orbits, and $B_{3}(t)$ corresponds to the core
part of the linearized system. Moreover we have
\bea
B_{3}(\theta)=\mathcal{B}_1(\theta)\diamond\cdots\diamond \mathcal{B}_{[\frac{n}{2}]}(\theta). \lb{bdcom} \eea
\end{thm}



\section{The Property of the Criteria Operator}

In this section, we first introduce the stability criteria via the Morse indices in Subsection 3.1.
This is based on the Maslov-type index theory described in \cite{Lon4} and the fact that the Maslov-type
index is essentially the same as the Morse index for second order Hamiltonian systems. In order to estimate
the Morse indices, we introduce the criteria operators with simple forms, and study their properties in
Subsections 3.1, 3.2, and 3.3.

\subsection{Stability criteria and the Morse indices of the corresponding operators}

Next we always define $\mathcal{A}(R,e)$ to be the linear operator corresponding to \eqref{st}, i.e.,
\bea
\mathcal{A}(R,e):=-\frac{d^2}{d\theta^2}I-2\mathbb{J}\frac{d}{d\theta}+r_e(\theta) R.
\eea
where $r_{e}(\theta)$ is defined in (\ref{msf}). Let
$$ \bar{D}_{n}(\omega,T)=\{y\in W^{2,2}([0,T],\mathbb{C}^{n})|y(T)=\omega y(0),\dot{y}(T)=\omega \dot{y}(0)\}. $$
Then $\mathcal{A}(R,e)$ is a self-adjoint operator in $L^2([0,2\pi], \mathbb{C}^{2n-2})$ with domain
$\bar{D}_{2n-2}(\omega,2\pi)$. We simply write it as $\mathcal{A}(R,e,\omega)$ and omit $\omega$ when
there is no confusion. It is obvious that if $R\leq D$, then $\mathcal{A}(R,e)\leq \mathcal{A}(D,e)$.
Here and below we write $A\le B$ for two linear symmetric operators $A$ and $B$, if $B-A\ge 0$, i.e., $B-A$
possesses no negative eigenvalues.

We define the $\omega$-Morse index $\phi_{\omega}(\mathcal{A}(R,e))$ to be the total number of negative
eigenvalues of $\mathcal{A}(R,e)$, and define
$\nu_{\omega}(\mathcal{A}(R,e))=\dim\ker(\mathcal{A}(R,e))$.

\begin{lem}\label{4.1} (See Long \cite{Lon4} p.172).
The $\omega$-Morse index $\phi_{\omega}(\mathcal{A}(R,e))$ and nullity $\nu_{\omega}(\mathcal{A}(R,e))$ are equal
to the $\omega$-Maslov-type index $i_{\omega}(\ga_e)$ and nullity $\nu_{\omega}(\ga_e)$ respectively, that is,
for any $\omega\in \mathbb{U}$, we have
\bea \phi_{\omega}(\mathcal{A}(R,e))=i_{\omega}(\ga_e),\ \ \nu_{\omega}(\mathcal{A}(R,e))=\nu_{\omega}(\ga_e). \label{eq}\eea
where $\ga_e$ is given by (\ref{ga3}).
\end{lem}

The next theorem follows from the corresponding property of the Maslov-type index.

\begin{thm}\label{4.1.1}(See (9.3.3) on p.204 of Long \cite{Lon4} with $\omega_0=-1$).
The matrix $\ga_e(2\pi)$ is spectral stable, if $|\phi_1(\mathcal{A}(R,e))- \phi_{-1}(\mathcal{A}(R,e))|=n$.
The matrix $\ga_e(2\pi)$ is hyperbolic, if $\mathcal{A}(R,e)$ is positive definite in $\bar{D}_{2n-2}(\omega,2\pi)$
for any $\omega\in\U$.
\end{thm}

We will estimate the $\omega$-Morse index $\phi_{\omega}(\mathcal{A}(R,e))$. Note first that from (\ref{bdcom}),
we have the following decomposition of the operator $\mathcal{A}(R,e)$,
\bea\label{4.76}
\mathcal{A}(R,e)=\mathcal{A}(R_1,e)\oplus\mathcal{A}(R_2,e)\oplus\cdots\oplus\mathcal{A}(R_{[\frac{n}{2}]},e),
\eea
and hence
\bea\label{4.8}
\phi_{\omega}(\mathcal{A}(R,e))=\sum_{l=1}^{[\frac{n}{2}]}\phi_{\omega}(\mathcal{A}(R_{l},e)).
\eea

Next, using notations defined in Section 2, we develop some techniques to estimate $\phi_{\omega}(\mathcal{A}(R_{l},e))$.
\\

\textbf{1)}\ \ For $l=1$, define $\check{d}_{n}=\min\{2P_{1},\frac{n}{2}\}$ and
$\hat{d}_{n}=\max\{2P_{1},\frac{n}{2}\}$. Let
\bea
E_{1}=\frac{n+m}{2}I_2,\ \ F_{1}=\frac{3\sqrt{m(m+n)}}{2}\mathcal{N},\ \ G_{1}=(2P_{1}+\frac{m}{2})I_2,  \label{4.9}
\eea where $\mathcal{N}= \left( \begin{array}{cccc} 1 & 0\\
0 & -1\end{array}\right)$.
Then
\bea
\check{\mathcal{U}}(1):=(\check{d}_{n}+\frac{m}{2})I_{4}+ \left(\begin{array}{cccc}0 & F_{1}\\
F_{1} & 0\end{array}\right)\leq \mathcal{U}(1)
     \leq \hat{\mathcal{U}}(1): =(\hat{d}_{n}+\frac{m}{2})I_{4}+\left(\begin{array}{cccc}0 & F_{1}\\
F_{1} & 0\end{array}\right).
\eea
Hence we have
\bea\label{3.7}
\mathcal{A}(\check{R}_1,e)\leq\mathcal{A}(R_1,e)\leq \mathcal{A}(\hat{R}_1,e),\ \ 
\eea
where
\bea\check{R}_1= I_4+\frac{1}{\lambda}\check{\mathcal{U}}(1),\quad \hat{R}_1= I_4+\frac{1}{\lambda}\hat{\mathcal{U}}(1). \eea
Let $T=\frac{1}{\sqrt{2}}\left(\begin{array}{cccc}I_{2} & I_{2}\\
-I_{2} & I_{2}\end{array}\right)$,
then direct computations yield
\bea\label{3.9}
&&T^t\mathcal{A}(\check{R}_1,e)T=\mathcal{A}(\check{R}^+_1,e)\oplus\mathcal{A}(\check{R}^-_1,e), \,\nonumber\\
&&T^t\mathcal{A}(\hat{R}_1,e)T=\mathcal{A}(\hat{R}^+_1,e)\oplus\mathcal{A}(\hat{R}^-_1,e),
\eea
where
\bea \check{R}^\pm_1= I_2+\frac{1}{\lambda}(\check{d}_n+ \frac{m}{2})I_{2}\pm \frac{3\sqrt{m(m+n)}}{2\lambda}\mathcal{N},\quad
\hat{R}^\pm_1= I_2+\frac{1}{\lambda}(\hat{d}_n+ \frac{m}{2})I_{2}\pm \frac{3\sqrt{m(m+n)}}{2\lambda}\mathcal{N}. \eea

It is easy to see that $\mathcal{A}(\check{R}^+_1,e)$ (or $\mathcal{A}(\hat{R}^+_1,e)$) is similar to
$\mathcal{A}(\check{R}^-_1,e)$ (or $\mathcal{A}(\hat{R}^-_1,e)$). Then we have
\bea
&&\phi_{\omega}(\mathcal{A}(\hat{R}_1,e))\leq\phi_{\omega}(\mathcal{A}(R_1,e))\leq
\phi_{\omega}(\mathcal{A}(\check{R}_1,e)),\nonumber \\
&&\phi_{\omega}(\mathcal{A}(\hat{R}_1,e))=2\phi_{\omega}(\mathcal{A}(\hat{R}^+_1,e)),\nonumber\\
&&\phi_{\omega}(\mathcal{A}(\check{R}_1,e))=2\phi_{\omega}(\mathcal{A}(\check{R}^+_1,e)).\label{4.19a}\label{}
\eea
\\

\textbf{2)}\ \ For $2\leq l \leq [\frac{n-1}{2}]$, by (\ref{2.17}), we define
\bea
E_{l}=\left( \begin{array}{cccc} a_{l} & 0\\
0 & b_{l}\end{array}\right), \ \ G_{l}=\left( \begin{array}{cccc} b_{l} & 0\\
0 & a_{l}\end{array}\right),\ \ \tilde{F}_{l}=S_{l} I_2. \eea
Then
\bea
\mathcal{U}(l)=
\frac{1}{2}\left(\begin{array}{cccc}E_{l}+G_{l} & -2S_{l}J_2\\
2S_{l}J_2 & E_{l}+G_{l}\end{array}\right)+\frac{1}{2}\left(\begin{array}{cccc}E_{l}-G_{l} & 0\\
0 & E_{l}-G_{l}\end{array}\right).
\eea
Then we obtain
\bea
\left(\begin{array}{cccc}E_{l}+G_{l}-2\tilde{F}_{l} & 0\\
0 & E_{l}+G_{l}-2\tilde{F}_{l}\end{array}\right)\leq\left(\begin{array}{cccc}E_{l}+G_{l} & -2S_{l}J_2\\
2S_{l}J_2 & E_{l}+G_{l}\end{array}\right)\leq\left(\begin{array}{cccc}E_{l}+G_{l}+2\tilde{F}_{l} & 0\\
0 & E_{l}+G_{l}+2\tilde{F}_{l}\end{array}\right).
\eea
Hence we have
\bea\label{3.16}
\mathcal{A}(\check{R}_{l},e)\leq\mathcal{A}(R_{l},e)\leq \mathcal{A}(\hat{R}_{l},e),\ \ 
\eea
where
\bea
&&\check{R}_{l}=
I_{4}+\frac{1}{2\lambda}\left(\begin{array}{cccc}E_{l}+G_{l}-2\tilde{F}_{l} & 0\\
0 & E_{l}+G_{l}-2\tilde{F}_{l}\end{array}\right)+\frac{1}{2\lambda}\left(\begin{array}{cccc}E_{l}-G_{l} & 0\\
0 & E_{l}-G_{l}\end{array}\right) \nonumber\\
&&\hat{R}_{l}=
I_{4}+\frac{1}{2\lambda}\left(\begin{array}{cccc}E_{l}+G_{l}+2\tilde{F}_{l} & 0\\
0 & E_{l}+G_{l}+2\tilde{F}_{l}\end{array}\right)+\frac{1}{2\lambda}\left(\begin{array}{cccc}E_{l}-G_{l} & 0\\
0 & E_{l}-G_{l}\end{array}\right)\nonumber .
\eea
Note that these two operators can be decomposed as follows,
\bea\label{3.17}
&&\mathcal{A}(\check{R}_{l},e)=\mathcal{A}(\check{R}_{l,0},e)\oplus\mathcal{A}(\check{R}_{l,0},e),\nonumber\\
&&\mathcal{A}(\hat{R}_{l},e)=\mathcal{A}(\hat{R}_{l,0},e)\oplus\mathcal{A}(\hat{R}_{l,0},e),
\eea
where
\bea\label{3.18}
\check{R}_{l,0}(\mu,e)&=&I_{2}+\frac{1}{2\lambda}(a_{l}+b_{l}-2S_{l})I_{2}
+\frac{1}{2\lambda}(a_{l}-b_{l})\mathcal{N}, \nonumber \\
\hat{R}_{l,0}(\mu,e)&=&I_{2}+\frac{1}{2\lambda}(a_{l}+b_{l}+2S_{l})I_{2}
+\frac{1}{2\lambda}(a_{l}-b_{l})\mathcal{N}.\label{4.26a}
\eea
Moreover we have
\bea
&&\phi_{\omega}(\mathcal{A}(\hat{R}_l,e))\leq\phi_{\omega}(\mathcal{A}(R_l,e))\leq
\phi_{\omega}(\mathcal{A}(\check{R}_l,e)),\nonumber \\
&&\phi_{\omega}(\mathcal{A}(\hat{R}_l,e))=2\phi_{\omega}(\mathcal{A}(\hat{R}_{l,0},e)),\nonumber\\
&&\phi_{\omega}(\mathcal{A}(\check{R}_l,e))=2\phi_{\omega}(\mathcal{A}(\check{R}_{l,0},e)).\label{4.27}
\eea

\textbf{3)} \ \ For $n\in2\mathbb{N}$, $l=[\frac{n}{2}]$, we have
\bea
R_{l}=I_{2}+\frac{1}{2\lambda}(a_{l}+b_{l})I_{2}
+\frac{1}{2\lambda}(a_{l}-b_{l})\mathcal{N}. \label{4.28a}
\eea

\subsection{The Criteria Operator}

Let
\bea  R_{\alpha, \beta}:=(1+\alpha)I_{2}+\beta\mathcal{N} \, \mathrm{with}\ \ \alpha\geq 0, \, \beta\geq 0. \nonumber\eea
From (\ref{4.19a}), (\ref{4.26a}), (\ref{4.28a}), we should estimate the $\omega$-Morse index of the following operator
\bea
\mathcal{A}(\alpha,\beta,e):=\mathcal{A}(R_{\alpha,\beta},e)\eea
whose domain is $\bar{D}_{2}(\omega,2\pi)$, and the corresponding Hamiltonian system of fundamental solution is given by
\bea
\dot{\gamma}_{\alpha, \beta, e}(\theta)=J_{4}B_{\alpha,\beta,e}(\theta)\gamma_{\alpha, \beta, e}(\theta),
     \quad\dot{\gamma}_{\alpha, \beta, e}(0)=I_{4},\nonumber
\eea
where
\bea
B_{\alpha,\beta,e}(\theta)=\left(\begin{array}{cccc}I_{2} & -J_{2}\\ J_{2} & I_{2}-r_e(\theta)R_{\alpha,\beta}\end{array}\right).\nonumber
\eea
From Lemma \ref{4.1}, we have
\bea
\phi_{\omega}(\mathcal{A}(\alpha,\beta,e))=i_{\omega}(\gamma_{\alpha,\beta,e}),\ \ \nu_{\omega}(\mathcal{A}(\alpha,\beta,e))
=\nu_{\omega}(\gamma_{\alpha,\beta,e}).\nonumber
\eea

\begin{rem}  The form $R_{\alpha,\beta}$ includes the case of three body problem. For the Lagrangian orbits, let $$ \delta=\frac{27(m_1m_2+m_2m_3+m_3m_1)}{(m_1+m_2+m_3)^2}\in[0,9].$$
Then the normalized Hessian with the form $R_{\alpha,\beta}$ for  $\alpha=\frac{1}{2}$, $\beta=\frac{\sqrt{9-\delta}}{2}$,
even includes the case of $\alpha$ potential, the details can be found in \cite{BJP}, \cite{HLS} and \cite{HS}. For the Euler orbits  \cite{HO}, \cite{ZL}, we have $R = diag(-\delta, 2\delta + 3)$, where $\delta \in [0, 7]$, only depends on mass $m_1, m_2, m_3$, and it can be given
explicitly in the form $R_{\alpha,\beta}$.
\end{rem}

Now we need the following lemma which is important in estimating the indices,

\begin{lem}\label{lem4.3} For $e\in[0,1)$, in space $\bar{D}_{2}(\omega,2\pi)$, we have
\bea
\mathcal{A}(\alpha,0,e)>0,\,  &\mathrm{if}&\ \   \alpha>0, \, \omega\in\mathbb{U},    \label{3.25}\\
\mathcal{A}(0,0,e)>0,\, &\mathrm{if}&\ \  \omega\in\mathbb{U}\setminus\{1\} \quad
            \mathrm{and} \quad \phi_1 (\mathcal{A}(0,0,e))=0, \,  \nu_1 (\mathcal{A}(0,0,e))=2, \label{3.26}\\
\phi_\omega(\mathcal{A}(1/2, 3/2,e)= 2,\,  &\mathrm{if}& \ \ \omega\in\mathbb{U}\setminus\{1\} \ \
        \mathrm{and}   \quad \phi_1(\mathcal{A}(1/2, 3/2,e)= 0,  \,  \nu_1(\mathcal{A}(1/2, 3/2,e)=3.  \label{3.27}\eea
\end{lem}

\begin{proof} Note that we have
\bea   \mathcal{A}(\alpha,0,e)=\mathcal{A}(\alpha, e)\oplus \mathcal{A}(\alpha, e), \nonumber \eea
where $\mathcal{A}(\alpha, e) = -\frac{d^{2}}{d\theta^{2}}-1+(1+\alpha) r_e(\theta)$.
From \cite{HO}, Proposition 3.2, $\mathcal{A}(\alpha, e)$ is positive definite in $\bar{D}_{1}(\omega,2\pi)$ for
$\alpha>0$ and $\omega\in\mathbb{U}$. When $\alpha=0$, from \cite{ZL}, Lemma 4.1, for $\omega=1$, in the space
$\bar{D}_{1}(\omega,2\pi)$, we have $\ker(\mathcal{A}(0, e))=\{c(1+e\cos(\theta))| c\in \mathbb{C}\}$, and
$\mathcal{A}(0, e)$ is positive definite for $\omega\neq1$. This implies \eqref{3.25} and \eqref{3.26}.

Please note that the case of $\alpha=1/2$ and $\beta=3/2$ corresponds to the linear system of Kepler orbits, and then
\eqref{3.27} is already proved in \cite{HLS} and \cite{HS}.
\end{proof}

Moreover, we have

\begin{prop}\label{pro4.2}
The $\omega$-Morse index $\phi_{\omega}(\mathcal{A}(\alpha,\beta,e))$ is decreasing in $\alpha\in[0,+\infty)$ and
it is increasing in $\beta\in(0,+\infty)$, when $\om$ and $e$ are fixed . Moreover, if $\alpha>0$, $\beta_{2}>\beta_{1}>0$,
and $\mathcal{A}(\alpha,\beta_{2},e)\geq0$, then $\mathcal{A}(\alpha,\beta_{1},e)>0$.
\end{prop}

\begin{proof}
When $\alpha_{1}>\alpha_{2}>0$, we have $\mathcal{A}(\alpha_{1},\beta,e)>\mathcal{A}(\alpha_{2},\beta,e)$
in $\bar{D}_{2}(\omega,2\pi)$. Hence
\bea
\phi_{\omega}(\mathcal{A}(\alpha_{1},\beta,e))\leq\phi_{\omega}(\mathcal{A}(\alpha_{2},\beta,e)).\nonumber
\eea
When $\beta_{2}\geq\beta_{1}>0$, let
\bea
\tilde{\mathcal{A}}(\alpha,\beta,e)=\frac{1}{\beta}\mathcal{A}(\alpha,0,e)+r_e\mathcal{N}.  \label{3.23}
\eea
Then we have
\bea\label{3.25s}
\phi_{\omega}(\mathcal{A}(\alpha,\beta,e))&=&\phi_{\omega}(\tilde{\mathcal{A}}(\alpha,\beta,e)),\nonumber\\
\nu_{\omega}(\mathcal{A}(\alpha,\beta,e))&=&\nu_{\omega}(\tilde{\mathcal{A}}(\alpha,\beta,e)).
\eea
Since $\mathcal{A}(\alpha,0,e)\geq0$ by \eqref{3.25} and \eqref{3.26},  we get
\bea
\tilde{\mathcal{A}}(\alpha,\beta_{2},e)\leq\tilde{\mathcal{A}}(\alpha,\beta_{1},e).\nonumber
\eea
Hence
\bea
\phi_{\omega}(\mathcal{A}(\alpha,\beta_{1},e))\leq\phi_{\omega}(\mathcal{A}(\alpha,\beta_{2},e)).\nonumber
\eea
Moreover, if $\alpha>0$, $\beta_{2}>\beta_{1}>0$, and $\mathcal{A}(\alpha,\beta_{2},e)\geq0$, from (\ref{3.25s}), we get
$\tilde{\mathcal{A}}(\alpha,\beta_{2},e)\geq0$. From (\ref{3.23}), we have
$
\tilde{\mathcal{A}}(\alpha,\beta_{1},e)-\tilde{\mathcal{A}}(\alpha,\beta_{2},e)=(\frac{1}{\beta_{1}}-\frac{1}{\beta_{2}})
\mathcal{A}(\alpha,0,e)
$.
Since $\mathcal{A}(\alpha,0,e)>0$ holds for $\alpha>0$ by \eqref{3.25} , we get
$
\tilde{\mathcal{A}}(\alpha,\beta_{1},e)>\tilde{\mathcal{A}}(\alpha,\beta_{2},e),
$
then $\tilde{\mathcal{A}}(\alpha,\beta_{1},e)>0$. Together with (\ref{3.25s}) it implies $\mathcal{A}(\alpha,\beta_{1},e)>0$.
\end{proof}


\begin{thm}\label{thm4.3}
For $\alpha\geq\frac{1}{2}$, $0<\beta< \alpha+1$, and $e\in [0,1)$, we have $\phi_{1}(\mathcal{A}(\alpha,\beta,e))=0,
\nu_{1}(\mathcal{A}(\alpha,\beta,e))=0$.
\end{thm}

\begin{proof}
Let $\alpha=\frac{1}{2}+\varepsilon$, $\varepsilon\geq 0$. Then
\bea
\mathcal{A}(\alpha,\alpha+1,e)=\mathcal{A}(\frac{1}{2},\frac{3}{2},e)+r_e\varepsilon(I_2+\mathcal{N}), \nonumber
\eea
and hence we get
\bea
\mathcal{A}(\alpha,\alpha+1,e)\geq\mathcal{A}(\frac{1}{2},\frac{3}{2},e). \nonumber
\eea
Together with \eqref{3.27},   
it yields $\mathcal{A}(\alpha,\alpha+1,e)\geq0$ in $\bar{D}_{2}(1,2\pi)$.
Since $\beta< \alpha+1$, from Proposition \ref{pro4.2}, we have $\mathcal{A}(\alpha,\beta,e)>0$ in
$\bar{D}_{2}(1,2\pi)$.
\end{proof}

On the other hand, in domain $\bar{D}_{2}(\omega,2\pi)$, $\mathcal{A}(\alpha,\beta,e)$ is similar to the operator
\bea
\bar{\mathcal{A}}(\alpha,\beta,e)=\mathcal{R}\mathcal{A}\mathcal{R}^T  =-\frac{d^2}{d\theta^2}I_{2}-\frac{d}{d\theta}I_{2}+r_e(\theta)((1+\alpha)I_{2}+
\beta \mathcal{R}(\theta)\mathcal{N}\mathcal{R}^{T}(\theta)).\nonumber
\eea
Hence
\bea
\phi_{\omega}(\mathcal{A}(\alpha,\beta,e))=\phi_{\omega}(\bar{\mathcal{A}}(\alpha,\beta,e)),\ \
\nu_{\omega}(\mathcal{A}(\alpha,\beta,e))=\nu_{\omega}(\bar{\mathcal{A}}(\alpha,\beta,e)).\nonumber
\eea
Now let
\bea
F(\beta,e)=2\beta r_e(\theta)\mathcal{R}(\theta)\left(\begin{array}{cccc}1 & 0\\ 0 & 0\end{array}\right)\mathcal{R}^{T}(\theta).\nonumber
\eea
Then it is easy to see that $F(\beta,e)\geq 0$ in $\bar{D}(\omega,2\pi)$ for any $\omega\in\mathbb{U}$ and
\bea
\bar{\mathcal{A}}(\alpha,\beta,e)=\bar{\mathcal{A}}(\alpha-\beta,0,e)+F(\beta,e),\nonumber
\eea
For any fixed $e_{0}\geq 0$, $\alpha_{0}>0$, and $\beta_{0}>0$, assume $e\geq e_{0}$. Then for any
$\omega\in\mathbb{U}$, we have
\bea
 \frac{\beta}{\beta_{0}}\frac{1+e_{0}}{1+e}F(\beta_{0},e_{0})\leq F(\beta,e)\leq \frac{\beta}{\beta_{0}}\frac{1-e_{0}}{1-e}F(\beta_{0},e_{0}).\nonumber
\eea
Hence we obtain
\bea
\bar{\mathcal{A}}(\alpha,\beta,e)&\geq&\bar{\mathcal{A}}(\alpha-\beta,0,e)+\frac{\beta}{\beta_{0}}\frac{1+e_{0}}{1+e}F(\beta_{0},e_{0})\nonumber\\
&=&\bar{\mathcal{A}}(\alpha-\beta,0,e)-\frac{\beta}{\beta_{0}}\frac{1+e_{0}}{1+e}\bar{\mathcal{A}}(\alpha_{0}-\beta_{0},0,e_{0})
+\frac{\beta}{\beta_{0}}\frac{1+e_{0}}{1+e}\bar{\mathcal{A}}(\alpha_{0},\beta_{0},e_{0}),\label{4.43}
\eea
and
\bea
\bar{\mathcal{A}}(\alpha,\beta,e)&\leq&\bar{\mathcal{A}}(\alpha-\beta,0,e)+\frac{\beta}{\beta_{0}}\frac{1-e_{0}}{1-e}F(\beta_{0},e_{0})\nonumber\\
&=&\bar{\mathcal{A}}(\alpha-\beta,0,e)-\frac{\beta}{\beta_{0}}\frac{1-e_{0}}{1-e}\bar{\mathcal{A}}(\alpha_{0}-\beta_{0},0,e_{0})
+\frac{\beta}{\beta_{0}}\frac{1-e_{0}}{1-e}\bar{\mathcal{A}}(\alpha_{0},\beta_{0},e_{0}).\label{4.44}
\eea

\begin{thm}
For $(\alpha,\beta,e)$ satisfying $0\leq e_{0}\leq e,1+\alpha_{0}-\beta_{0}>0$, $\alpha_{0}>0, \beta_{0}>0, \alpha>0, \beta\geq 0$,
we have

i)\ \ If
\bea \frac{\beta}{\beta_0} \frac{1+e_{0}}{1+e}<1, \quad
 \frac{\beta(e-e_0)}{\beta_0(1+e)}< \alpha-\frac{\beta}{\beta_0}\alpha_0 ,  \label{th4.5f1}\eea
then
\bea
\frac{\beta}{\beta_{0}}\frac{1+e_{0}}{1+e}\mathcal{A}(\alpha_{0},\beta_{0},e_{0})<\mathcal{A}(\alpha,\beta,e),\ \
\phi_{\omega}(\mathcal{A}(\alpha,\beta,e))\leq\phi_{\omega}(\mathcal{A}(\alpha_{0},\beta_{0},e_{0})).\nonumber
\eea

ii)\ \ If
\bea \frac{\beta}{\beta_0} \frac{1-e_{0}}{1-e}>1, \quad
\frac{\beta(e_0-e)}{\beta_0(1-e)}>\alpha-\frac{\beta}{\beta_0}\alpha_0 ,  \label{th4.5f2}  \eea
then
\bea
\mathcal{A}(\alpha,\beta,e)<\frac{\beta}{\beta_{0}}\frac{1-e_{0}}{1-e}\mathcal{A}(\alpha_{0},\beta_{0},e_{0}),
      \ \ \phi_{\omega}(\mathcal{A}(\alpha,\beta,e))\geq\phi_{\omega}(\mathcal{A}(\alpha_{0},\beta_{0},e_{0})).\nonumber
\eea
\label{thm4.5}
\end{thm}

\begin{proof}
i)\ \ For $1+\alpha_{0}-\beta_{0}>0, \frac{\beta}{\beta_{0}}\frac{1+e_{0}}{1+e}\neq 1$, we have
\bea
&&\bar{\mathcal{A}}(\alpha-\beta,0,e)-\frac{\beta}{\beta_{0}}\frac{1+e_{0}}{1+e}\bar{\mathcal{A}}(\alpha_{0}-\beta_{0},0,e_{0})\nonumber\\
&& \geq(1-\frac{\beta}{\beta_{0}}\frac{1+e_{0}}{1+e})
(-\frac{d^2}{d\theta^2}I_{2}-\frac{d}{d\theta}I_{2}+\frac{1+\alpha-\beta-\frac{\beta}{\beta_{0}}(1+\alpha_{0}-\beta_{0})}
{1-\frac{\beta}{\beta_{0}}\frac{1+e_{0}}{1+e}}r_e(\theta)I_2 ).\nonumber
\eea

If $\frac{1+\alpha-\beta-\frac{\beta}{\beta_{0}}(1+\alpha_{0}-\beta_{0})}
{1-\frac{\beta}{\beta_{0}}\frac{1+e_{0}}{1+e}}>1$ and $\frac{\beta}{\beta_{0}}\frac{1+e_{0}}{1+e}< 1$, then from  \eqref{3.25}, we have
\bea
\bar{\mathcal{A}}(\alpha-\beta,0,e)-\frac{\beta}{\beta_{0}}\frac{1+e_{0}}{1+e}\bar{\mathcal{A}}(\alpha_{0}-\beta_{0},0,e_{0})>0
\ \ \mathrm{in}\ \ \bar{D}_{2}(\omega,2\pi), \forall \omega\in\mathbb{U}. \nonumber
\eea
Together with (\ref{4.43}), it yields
\bea
\bar{\mathcal{A}}(\alpha,\beta,e)>\frac{\beta}{\beta_{0}}\frac{1+e_{0}}{1+e}\bar{\mathcal{A}}(\alpha_{0},\beta_{0},e),
\ \ \mathrm{in}\ \ \bar{D}_{2}(\omega,2\pi), \forall \omega\in\mathbb{U},\nonumber
\eea

ii)\ \ For $1+\alpha_{0}-\beta_{0}>0$ and $\frac{\beta}{\beta_{0}}\frac{1+e_{0}}{1+e}\neq 1$, we have
\bea
&&\bar{\mathcal{A}}(\alpha-\beta,0,e)-\frac{\beta}{\beta_{0}}\frac{1-e_{0}}{1-e}\bar{\mathcal{A}}(\alpha_{0}-\beta_{0},0,e_{0})\nonumber\\
&& \leq (1-\frac{\beta}{\beta_{0}}\frac{1-e_{0}}{1-e})
(-\frac{d^2}{d\theta^2}I_{2}-\frac{d}{d\theta}I_{2}+\frac{1+\alpha-\beta-\frac{\beta}{\beta_{0}}(1+\alpha_{0}-\beta_{0})}
{1-\frac{\beta}{\beta_{0}}\frac{1-e_{0}}{1-e}}r_e(\theta)I_2).\nonumber
\eea

If $\frac{1+\alpha-\beta-\frac{\beta}{\beta_{0}}(1+\alpha_{0}-\beta_{0})}{1-\frac{\beta}{\beta_{0}}\frac{1-e_{0}}{1-e}}>1$
and $\frac{\beta}{\beta_{0}}\frac{1-e_{0}}{1-e}> 1$, then from\eqref{3.25}, we have
\bea
\bar{\mathcal{A}}(\alpha-\beta,0,e)-\frac{\beta}{\beta_{0}}\frac{1-e_{0}}{1-e}\bar{\mathcal{A}}(\alpha_{0}-\beta_{0},0,e_{0})<0
\ \ \mathrm{in}\ \ \bar{D}_{2}(\omega,2\pi), \forall \omega\in\mathbb{U}.\nonumber
\eea
Together with (\ref{4.44}), it yields
\bea
\bar{\mathcal{A}}(\alpha,\beta,e)<\frac{\beta}{\beta_{0}}\frac{1-e_{0}}{1-e}\bar{\mathcal{A}}(\alpha_{0},\beta_{0},e),
\ \ \mathrm{in}\ \ \bar{D}_{2}(\omega,2\pi), \forall \omega\in\mathbb{U}. \nonumber
\eea
\end{proof}

These theorems tell us that if we know that the $\omega$-Morse index of $\mathcal{A}(\alpha_{0},\beta_{0},e_{0})$
for some $(\alpha_{0},\beta_{0},e_{0})$ satisfies the corresponding conditions, then we can get the upper and lower
bounds of the $\omega$-Morse index of $\mathcal{A}(\alpha,\beta,e)$ for some $(\alpha,\beta,e)$ related to
$(\alpha_{0},\beta_{0},e_{0})$. In the case $e_{0}=0$, we can compute the fundamental solution
$\gamma_{\alpha_{0},\beta_{0},0}(2\pi)$ directly. Moreover we can also compute the indices
$\phi_{1}(\mathcal{A}(\alpha_{0},\beta_{0},0))$ and $\phi_{-1}(\mathcal{A}(\alpha_{0},\beta_{0},0))$ for
$\alpha_{0}\geq0,\beta_{0}\geq0$, then we can use them to estimate the Morse indices
$\phi_{1}(\mathcal{A}(\alpha,\beta,e))$ and $\phi_{-1}(\mathcal{A}(\alpha,\beta,e))$ for $e>0$.  In the next section,
we will compute the $-1$ and $1$-Morse indices of the operator $\mathcal{A}(\alpha_{0},\beta_{0},0)$.


\subsection{Computation of $-1$ and $1$-Morse index of operator $\mathcal{A}(\alpha,\beta,0)$} 

Simple computations show that
\bea
\sigma(J_{4}B_{\alpha,\beta,0}(\theta))=\left\{\pm\(\alpha-1+(\beta^2-4\alpha)^{\frac{1}{2}}\)^{\frac{1}{2}},
\pm\(\alpha-1-(\beta^2-4\alpha)^{\frac{1}{2}}\)^{\frac{1}{2}}\right\},\, \quad \mathrm{for}\;\;\alpha\geq0,\, \beta\geq0. \nonumber
\eea
Then we have
\\
i)\ \  $1\in\sigma(\gamma_{\alpha,\beta,0}(2\pi))$, if and only if
\bea
\(\alpha-1\pm(\beta^2-4\alpha)^{\frac{1}{2}}\)^{\frac{1}{2}}=k\sqrt{-1},\,\quad  k\in\mathbb{Z}, \nonumber
\eea
especially, we have
\bea
&&\beta=\alpha+1,\ \ \alpha\neq\frac{1}{2}, \ \ \nu_{1}(\gamma_{\alpha,\beta,0}(2\pi))=1,\nonumber \\
&&\beta=(\alpha^2+4\alpha)^{\frac{1}{2}},\ \ \alpha\neq\frac{1}{2},\ \ \nu_{1}(\gamma_{\alpha,\beta,0}(2\pi))=2,
\nonumber \\
&&\beta=\frac{3}{2},\ \ \alpha=\frac{1}{2},\ \ \nu_{1}(\gamma_{\alpha,\beta,0}(2\pi))=3.
\nonumber
\eea
Since $\mathcal{A}(\alpha,0,0)\geq 0$ in $\bar{D}(\omega,2\pi)$ for any $\omega\in\U$, so
$\phi_{1}(\mathcal{A}(\alpha,0,0))=0$ holds. Together with Proposition \ref{pro4.2}, it yields
\bea
&&\phi_{1}(\mathcal{A}(\alpha,\beta,0))=0, \ \ \mathrm{for}\ \
    (\alpha,\beta)\in \mathcal{D}_{1}=\{(x,y)|x\geq0,0\leq y\leq\min\{x+1,(x^2+4x)^{\frac{1}{2}}  \}. \nonumber
\eea
Moreover, we have the picture of the $1$-degenerate curves and the distribution of
$\phi_{1}(\mathcal{A}(\alpha,\beta,0))$ in Figure 1.
\begin{figure}[H]
  \centering
   \includegraphics[height=0.40\textwidth,width=0.68\textwidth]{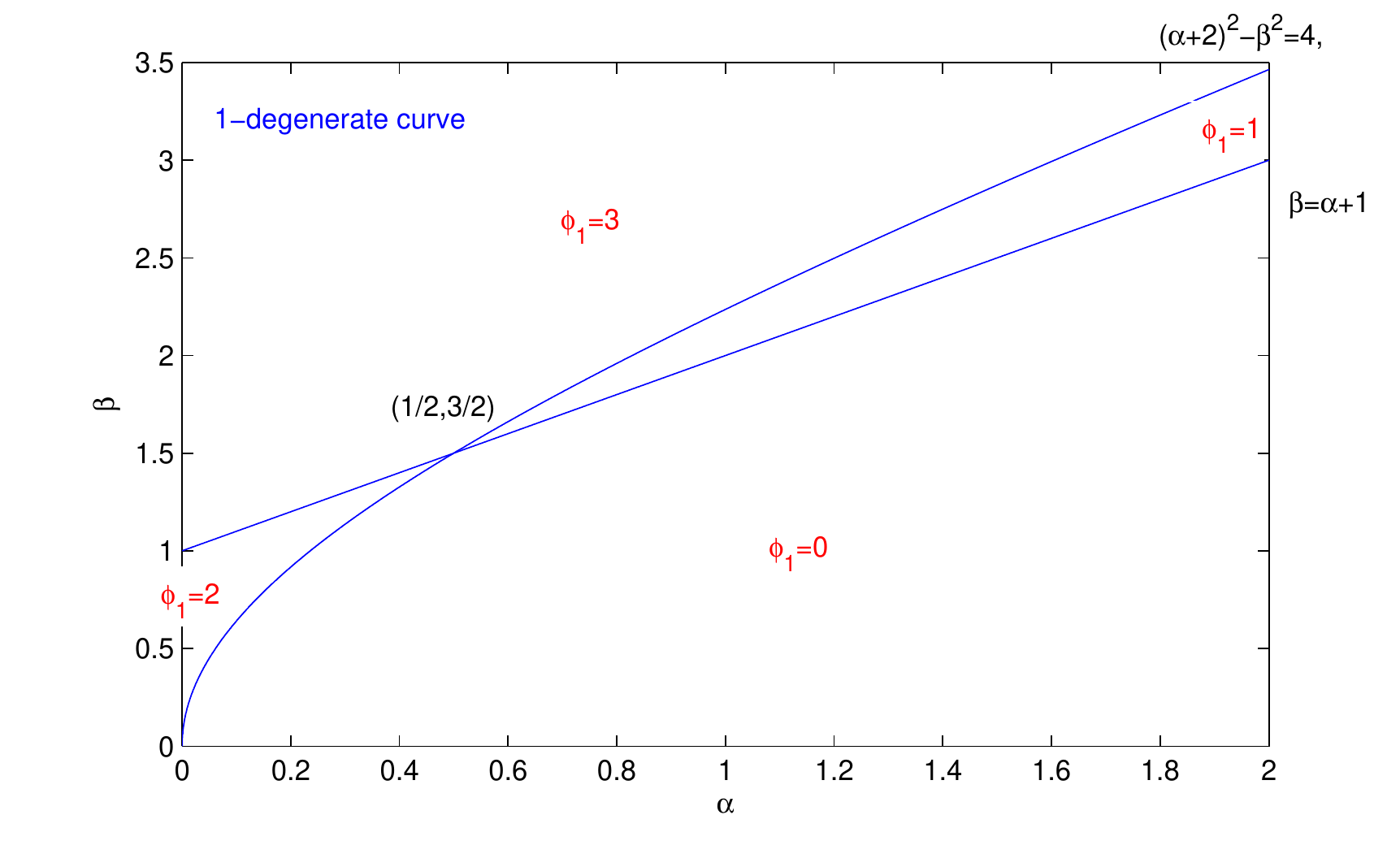}
     \caption{$1$-degenerate cures for $k=0$ and $k=\pm1$}.
\end{figure}
ii)\ \ $-1\in\sigma(\gamma_{\alpha,\beta,0}(2\pi))$, if and only if
\bea
\(\alpha-1\pm(\beta^2-4\alpha)^{\frac{1}{2}}\)^{\frac{1}{2}}=(\pm\frac{1}{2}+k)\sqrt{-1},\, k\in\mathbb{Z}, \nonumber
\eea
especially, we have
\bea
&&\beta=\(\alpha^2+\frac{5}{2}\alpha+\frac{9}{16}\)^{\frac{1}{2}},\, \nu_{-1}(\gamma_{\alpha,\beta,0}(2\pi))=2,\ \ \mathrm{for}\ \ k=0 \ \ \mathrm{or}\ \ k=-1,\nonumber \\
&&\beta=\(\alpha^2+\frac{13}{2}\alpha+\frac{25}{16}\)^{\frac{1}{2}},\, \nu_{-1}(\gamma_{\alpha,\beta,0}(2\pi))=2,\ \ \mathrm{for}\ \ k=1\ \ \mathrm{or}\ \ k=-2.\nonumber
\eea
Since $\mathcal{A}(\alpha,0,0)\geq 0$ in $\bar{D}(\omega,2\pi)$ for any $\omega\in\U$, so $\phi_{-1}(\mathcal{A}(\alpha,0,0))=0$.
Together with Proposition \ref{pro4.2}, it yields
\bea
&&\phi_{-1}(\mathcal{A}(\alpha,\beta,0))=0, \ \
  \mathrm{for}\ \ (\alpha,\beta)\in \mathcal{E}_{1}=\left\{(x,y)|x\geq0,0\leq y\leq\(x^2+\frac{5}{2}x+\frac{9}{16}\)^{\frac{1}{2}}\right\}. \nonumber
\eea
Moreover, we have the picture of the $-1$-degenerate curves and the distribution of
$\phi_{-1}(\mathcal{A}(\alpha,\beta,0))$ in Figure 2.
\begin{figure}[H]
  \centering
   \includegraphics[height=0.40\textwidth,width=0.68\textwidth]{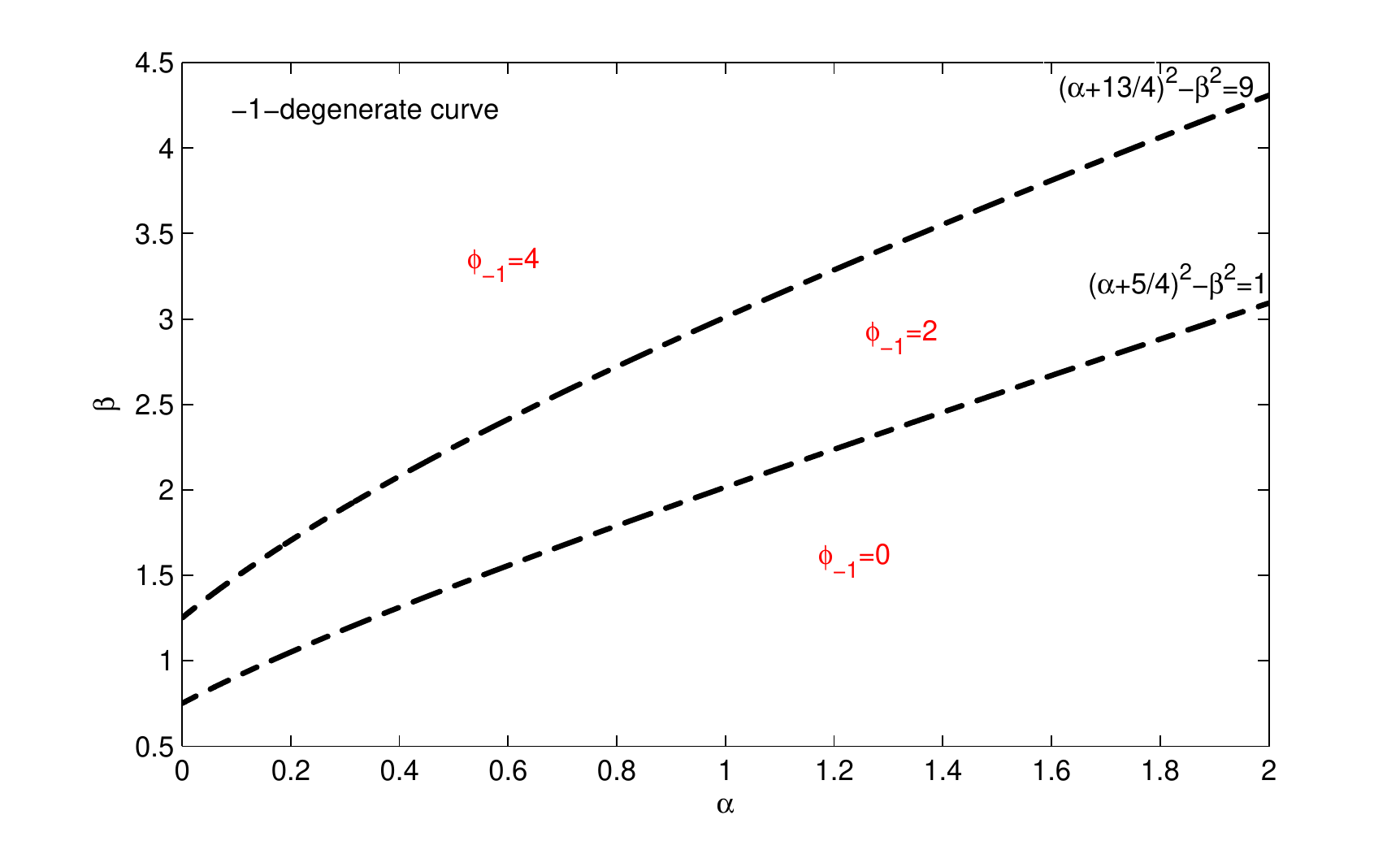}
     \caption{$-1$-degenerate cures for $k=0,-1$ or $k=1,-2$.}
\end{figure}
iii)\ \ $\gamma_{\alpha,\beta,0}(2\pi)$ is hyperbolic, if and only if
\bea
\beta<2\sqrt{\alpha},\ \ \mathrm{or}\ \ 2\sqrt{\alpha}\leq\beta<\alpha+1\;\;\mathrm{and}\;\;\alpha>1. \nonumber
\eea
Combining Figures 1 and 2 together yields the Figure 3.
\begin{figure}[H]
  \centering
   \includegraphics[height=0.40\textwidth,width=0.68\textwidth]{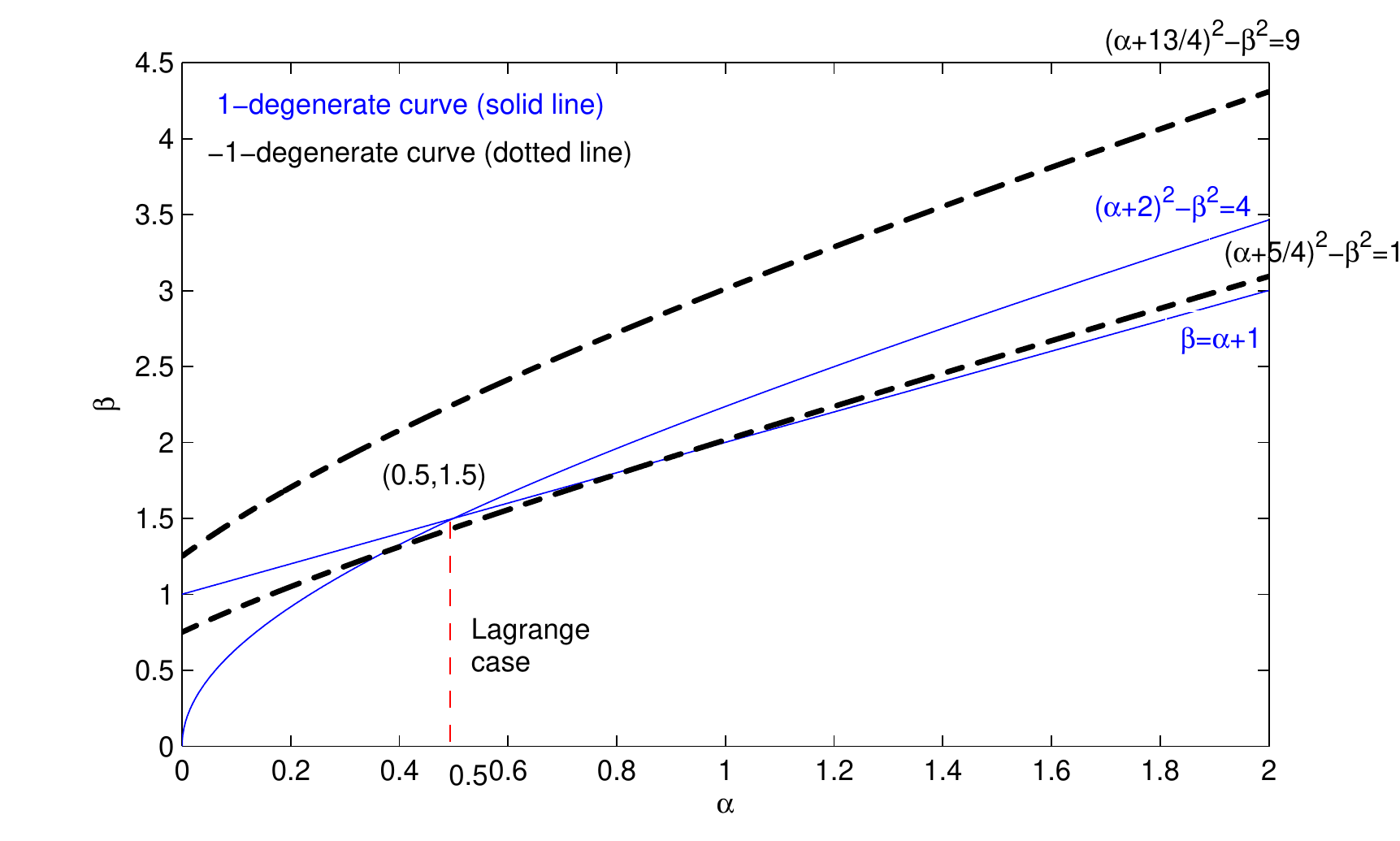}
     \caption{$1$-degenerate cures and $-1$-degenerate curves}.
\end{figure}

\begin{cor}\label{cor4.6}
The case of $\alpha=\frac{1}{2},\beta\in[0,\frac{3}{2}]$ and $e=0$ corresponds to the circular Lagrangian
solutions, and
\bea
&&\phi_{1}(\mathcal{A}(\frac{1}{2},\beta,0))=0, \ \qquad \forall\beta\in [0,\frac{3}{2}],\nonumber\\
&&\phi_{-1}(\mathcal{A}(\frac{1}{2},\beta,0))=2, \ \qquad \forall\beta\in(\frac{\sqrt{33}}{4},\frac{3}{2}]. \nonumber
\eea
\end{cor}

\begin{rem}\label{rem4.7}
In fact, in \cite{HLS,HO}, much stronger results were proved.
\bea
&&\phi_{1}(\mathcal{A}(\frac{1}{2},\beta,e))=0,\ \ \ \ \forall(\beta,e)\in [0,\frac{3}{2}]\times[0,1), \nonumber\\
&&\phi_{\omega}(\mathcal{A}(\frac{1}{2},\beta,e))=0, \nu_{\omega}(\mathcal{A}(\frac{1}{2},\beta,e))=0,\ \
       \ \forall(\beta,e)\in \mathfrak{U}_{1},\omega\in\mathbb{U},\nonumber\\
&&\phi_{-1}(\mathcal{A}(\frac{1}{2},\beta,e))=2, \nu_{-1}(\mathcal{A}(\frac{1}{2},\beta,e))=0,
\ \ \forall \beta\in (\beta_{m}(e),\frac{3}{2}],e\in[0,1),\nonumber
\eea
where $\mathfrak{U}_{1}=\{(\beta,e)| 0\leq\beta<\min\{\frac{\sqrt{9-x_{0}}}{2},\frac{(1+e)\sqrt{9-x_{0}}}{2(1+3e-2y_{0})}\}$,
$e\in[0,1)\}$, $(x_{0},y_{0})=(1.5,0.108)$, and $\beta_{m}(e)>0$ is an analytic curve in $e$.
\end{rem}


\begin{cor}\label{cor4.9}
Let $\alpha_{0}=\frac{1}{2}$, $\beta_{0}\in(0,\frac{3}{2}]$ and $e_{0}\in[0,1)$, then Remark \ref{rem4.7} and
Theorem \ref{thm4.5} imply that if $(\alpha,\beta,e)$ satisfies
\bea
e_{0}\leq e,\quad    \frac{\beta}{\beta_{0}}\frac{1+e_{0}}{1+e}< 1,\quad
\alpha>\frac{\beta}{\beta_{0}}\frac{1+3e-2e_{0}}{2+2e},
\nonumber
\eea
then
\bea
\mathcal{A}(\alpha,\beta,e)>0\ \ \mathrm{in}\ \ \bar{D}_{1}(1,2\pi).\nonumber
\eea
If we choose $\beta_{0}=\frac{3}{2}$ and $e_{0}\in [0,1)$, then
\bea
\mathcal{A}(\alpha,\beta,e)>0\ \ \mathrm{in}\ \ \bar{D}_{1}(1,2\pi), \ \ \mathrm{for}\ \ e\in[e_{0},1), \  \ \beta\in [0, \frac{3}{2}\frac{1+e}{1+e_{0}}), \ \ \alpha\in(\frac{1+3e-2e_{0}}{3+3e}\beta,+\infty). \nonumber
\eea
Moreover, if we take $e=e_{0}$, then
\bea
\mathcal{A}(\alpha,\beta,e)>0\ \ \mathrm{in}\ \ \bar{D}_{1}(1,2\pi), \  \ \mathrm{for}\ \ \beta\in [0, \frac{3}{2}), \ \ \alpha\in(\frac{1}{3}\beta,+\infty). \nonumber
\eea
\end{cor}
\begin{cor}\label{cor4.10}
Choose $\alpha_{0}=\frac{1}{2}, (\beta_{0},e_{0})\in\mathfrak{U}_{1}$. Then Remark \ref{rem4.7} and Theorem \ref{thm4.5} imply that
if $(\alpha,\beta,e)$ satisfies
\bea
e_{0}\leq e,\quad  \frac{\beta}{\beta_{0}}\frac{1+e_{0}}{1+e}< 1,\quad
      \alpha>\frac{\beta}{\beta_{0}}\frac{1+3e-2e_{0}}{2+2e},
\nonumber
\eea
then
\bea
\mathcal{A}(\alpha,\beta,e)>0\ \ \mathrm{in}\ \ \bar{D}_{1}(\omega,2\pi)\ \ \ \forall \omega\in\mathbb{U}.\nonumber
\eea
where $\mathfrak{U}_{1}=\{(\beta_{0},e_{0})|\ \ 0<\beta_{0}<\min\{\frac{\sqrt{9-x_{0}}}{2}$,
$\frac{(1+e_{0})\sqrt{9-x_{0}}}{2(1+3e_{0}-2y_{0})}, e_{0}\in[0,1)\},(x_{0},y_{0})=(1.5,0.108)$.
Simple computations show that
$(0,0.7237]\times[0,1)\subset(0,\frac{\sqrt{9-x_{0}}}{4-2y_{0}})\times[0,1)\subset\mathfrak{U}_{1}$. Especially if we let
$e=e_{0}$, then
\bea
\mathcal{A}(\alpha,\beta,e)>0\ \ \mathrm{in}\ \ \bar{D}_{1}(\omega,2\pi)\ \ \ \ \forall \omega\in\mathbb{U},
    \;(\beta_{0},e)\in(0,0.7237]\times[0,1),\;
\beta\in [0, \beta_{0}),\;\alpha\in(\frac{\beta}{2\beta_{0}},+\infty). \nonumber
\eea
\end{cor}

\section{The Stability of $(1+n)$-gon ERE}

In this section, we will estimate the $\pm 1$-Morse indices in Subsections 1 and 2 respectively, and give
the proof of Theorem \ref{th1.1}.

\subsection{Estimate $1$-Morse index of $(1+n)$-gon ERE}\label{sec4.5}

\textbf{1)}\ \ For $l=1$, from (\ref{3.7}), (\ref{3.9}) we have
\bea \mathcal{A}(R_1,e)\geq
\mathcal{A}(\check{R}_1,e), \ \
T^{t}\mathcal{A}(\check{R}_1,e)T=\mathcal{A}(\check{R}^{+}_1,e)\oplus\mathcal{A}(\check{R}^{-}_1,e),\nonumber
\eea
and $\mathcal{A}(\check{R}^{-}_1,e)$ is similar to $\mathcal{A}(\check{R}^{+}_1,e)=\mathcal{A}(\alpha_1, \beta_{1}, e)$,
where $\alpha_1=\frac{1}{\lambda}(\check{d}_{n}+\frac{m}{2})$, $\beta_1=\frac{3\sqrt{m(m+n)}}{2\lambda}$
with $\lambda=m+\frac{1}{2}\sigma_{n}$, $\check{d}_{n}=\min\{2P_{1},\frac{n}{2}\},\hat{d}_{n}=\max\{2P_{1},\frac{n}{2}\}$,
$\sigma_{n}=\frac{1}{2}\sum_{i=1}^{n-1}\csc \frac{\pi i}{n}$ and $P_{1}=\sum_{j=1}^{n-1}\frac{1-\cos^2 \theta_{j}}{2d_{nj}^3}$.

If
\bea
\frac{\check{d}_{n}+\frac{m}{2}}{m+\frac{\sigma_{n}}{2}}\geq\frac{1}{2},\ \
\frac{\frac{3}{2}\sqrt{m(m+n)}}{m+\frac{\sigma_{n}}{2}}< \frac{\check{d}_{n}+\frac{m}{2}}{m+\frac{\sigma_{n}}{2}}+1, \label{4.74}
\eea
then by Theorem \ref{thm4.3}, we have $\mathcal{A}(\alpha_1,\beta_1,e)>0$ in $\bar{D}_{1}(1,2\pi)$, for all $e\in [0,1)$.
Inequalities in (\ref{4.74}) are implied by the following inequalities
\bea
4\check{d}_{n}\geq \sigma_{n},\ \ \frac{4}{3}\check{d}_{n}+\frac{2}{3}\sigma_{n}> n. \label{4.75}
\eea
By using the Matlab, we can compute $\sigma_{n}$ and $\check{d}_{n}$ directly for $4\leq n\leq 27$. We list
the numerical results below
\begin{center}
\begin{tabular}{|c|c|c|c|c|c|c|c|c|c|c|}
  \hline
  $\sigma_{4}\sim \sigma_{11} $ & 1.9142 & 2.7528 & 3.6547 & 4.6095 & 5.6097 & 6.6497 & 7.7249 & 8.8319  \\
  $\sigma_{12}\sim \sigma_{19}$ & 9.9679 & 11.1304 & 12.3173 & 13.5269 & 14.7578 & 16.0085 & 17.2780 & 18.5652\\
  $\sigma_{20}\sim \sigma_{27}$ & 19.8690 & 21.1889 & 22.5238 & 23.8732 & 25.2365 & 26.6130 & 28.0023 & 29.4038 \\
  \hline
\end{tabular}
\end{center}
\begin{center}
\begin{tabular}{|c|c|c|c|c|c|c|c|c|c|c|}
  \hline
  $\check{d}_{4}\sim \check{d}_{11} $ & 0.7072 & 1.2140 & 1.7886 & 2.4188 & 3.0960 & 3.8140 & 4.5680 & 5.3544  \\
  $\check{d}_{12}\sim \check{d}_{19}$ & 6 & 6.5 & 7 & 7.5 & 8 & 8.5 & 9 & 9.5\\
  $\check{d}_{20}\sim \check{d}_{27}$ & 10 & 10.5 & 11 & 11.5 & 12 & 12.5 & 13 & 13.5 \\
  \hline
\end{tabular}
\end{center}

We find that they satisfy the inequalities (\ref{4.75}) for $9\leq n \leq 27$. Hence
$\mathcal{A}(\alpha_1,\beta_1,e)>0$ in $\bar{D}_{1}(1,2\pi)$, this implies that
$\mathcal{A}(R_1,e)>0$ in $\bar{D}_{2}(1,2\pi)$ holds and we get the following lemma.

\begin{lem}\label{lem4.11}
For $9\leq n \leq 27$, the equality $\mathcal{A}(R_1,e)>0$ in $\bar{D}_{2}(1,2\pi)$ holds for all $(m ,e)\in [0,+\infty)\times [0,1)$.
\end{lem}
Moreover, if
\bea
n<\sigma_{n},\quad  n\leq4\check{d}_{n},\label{4.70}
\eea
then
\bea
\frac{\frac{3}{2}\sqrt{m(m+n)}}{m+\frac{1}{2}\sigma_{n}}<\frac{3}{2},\ \
        \frac{1}{3}<\frac{\check{d}_{n}+\frac{m}{2}}{\frac{3}{2}\sqrt{m(m+n)}}.
\eea
By Corollary \ref{cor4.9}, we have
\bea
\mathcal{A}(\alpha_1,\beta_1,e)>0\ \ \mathrm{in}\ \ \bar{D}_{1}(1,2\pi),\quad \mathrm{for}\ \ \forall(m,e)\in(0,+\infty)\times[0,1),\nonumber
\eea
hence
\bea
\mathcal{A}(R_{1},e)>0\ \ \mathrm{in}\ \ \bar{D}_{2}(1,2\pi),\quad \mathrm{for}\ \ \forall(m,e)\in(0,+\infty)\times[0,1).\nonumber
\eea
We need to find an integer $n_{0}\ge 0$ such that for any $n\geq n_{0}$, the inequality (\ref{4.70}) holds.

\begin{lem}\label{lem4.12}
When $n\ge 28$ we have
\bea
\mathcal{A}(R_{1},e)>0\ \ \mathrm{in}\ \ \bar{D}_{2}(1,2\pi),\quad \mathrm{for}\ \ \forall(m,e)\in(0,+\infty)\times[0,1).\nonumber
\eea
\end{lem}

\begin{proof}
Since $\check{d}_{n}=\min\{2P_{1},\frac{n}{2}\}$, from the inequality (\ref{4.70}), we only need to find
$n_{0}$ such that $2P_{1}\geq\frac{n}{4}$ and $\sigma_{n}\geq n$ for all $n\geq n_{0}$. By using the
inequality $\sin x\leq x \leq \tan x$ for $x\in[0,\frac{\pi}{2}]$, we have
\bea
&& \sigma_{n}\geq\sum_{i=1}^{[\frac{n}{2}]-1}\frac{1}{\sin \frac{\pi i}{n}}\geq \sum_{i=1}^{[\frac{n}{2}]-1}\frac{n}{\pi i},\nonumber\\
&& 2P_{1}=\sigma_{n}-\frac{1}{2}\cot\frac{\pi}{2n}\geq\sum_{i=1}^{[\frac{n}{2}]-1}\frac{n}{\pi i}-\frac{n}{\pi}.\nonumber
\eea
Hence we only need to find $n_{0}$ such that
\bea
\sum_{i=1}^{[\frac{n_{0}}{2}]-1}\frac{1}{\pi i}\geq 1.\nonumber
\eea
Now it is easy to check that if $n_{0}=28$, then
\bea
\sum_{i=1}^{[\frac{n_{0}}{2}]-1}\frac{1}{\pi i}\approx 1.0123 \geq 1,\nonumber
\eea
which yields the lemma.
\end{proof}

\begin{rem}\label{remAA}
Note that the above method doesn't work for the cases $n=7$ or $8$.
The reason is that we have used the operator $\mathcal{A}(\check{R}_{1},e)$ to give the lower
bound for the original operator $\mathcal{A}(R_{1},e)$. Since $\mathcal{A}(\check{R}_{1},e)$
can be decomposed, it is much simpler to give its estimates. For $n\geq 9$, the operator
$\mathcal{A}(\check{R}_{1},e)$ is positive definite in $\bar{D}_{1}(1,2\pi)$, but it is not so for $n=7$ or $8$. In
fact, even for $e=0$ and $m$ being large enough, this operator is not positive definite. Hence,
in order to get some similar results for $n=7$ or $8$, it is necessary to study the original
operator $\mathcal{A}(R_{1},e)$. \end{rem}

By using some local methods, we can get the following lemma for $n=8$, which leads to the
stability result too. But it seems not work for $n=7$.

\begin{lem}\label{lem4.13}
There exist a function $m_{1}(e)>0$ depending on $e$ such that when $n=8$ the following holds,
\bea
\mathcal{A}(R_{1},e)>0\ \ \mathrm{in}\ \ \bar{D}_{2}(1,2\pi),\quad \mathrm{for}\ \ \forall(m,e)\in(m_{1}(e),+\infty)\times[0,1).\nonumber
\eea
\end{lem}

\begin{proof}
The operator $\mathcal{A}(R_{1},e)$ is similar to the operator
$\widetilde{\mathcal{A}}_{1}(\eta,e)=T^{t}\mathcal{A}(R_{1},e)T$, where
$T=\frac{1}{\sqrt{2}}\left(
     \begin{array}{cc}
        I_{2} & I_{2} \\
       -I_{2} & I_{2} \\
         \end{array}
        \right)$, and $\eta=\frac{1}{m}$. Then we have
\bea
\widetilde{\mathcal{A}}_{1}(\eta,e)=
-\frac{d^2}{d\theta^2}I_{4}-2\mathbb{J}_{2}\frac{d}{d\theta}
+r_e(\theta)\left(I_{4}+\alpha(\eta,n)I_{4}+\beta(\eta,n)\left(\begin{array}{cccc}-\mathcal{N} & O_{2}\\
O_{2} & \mathcal{N}\end{array}\right)+\gamma(\eta,n)\left(\begin{array}{cccc}I_{2} & I_{2}\\
I_{2} & I_{2}\end{array}\right)\right),\nonumber
\eea
where $O_{2}$ is the $2\times 2$ zero matrix, $\alpha(\eta,n)=\frac{4P_{1}\eta+1}{\sigma_{n}\eta+2}$,
$\beta(\eta,n)=\frac{3\sqrt{1+n\eta}}{\sigma_{n}\eta+2}$, and $\gamma(\eta,n)=\frac{(n/2-2P_{1})\eta}{\sigma_{n}\eta+2}$.
When $\eta=0$,
\bea
\widetilde{\mathcal{A}}_{1}(\eta,e)&=&
-\frac{d^2}{d\theta^2}I_{4}-2\mathbb{J}_{2}\frac{d}{d\theta}
+r_e(\theta)\left(I_{4}+\frac{1}{2}I_{4}+\frac{3}{2}\left(\begin{array}{cccc}-\mathcal{N} & O_{2}\\
O_{2} & \mathcal{N}\end{array}\right)\right)\nonumber\\
&=&\mathcal{A}(\frac{1}{2},\frac{3}{2},e)\oplus\mathcal{A}(\frac{1}{2},-\frac{3}{2},e),\nonumber
\eea
and $J^{T}_{2}\mathcal{A}(\frac{1}{2},-\frac{3}{2},e)J_{2}=\mathcal{A}(\frac{1}{2},\frac{3}{2},e)$, hence
$\widetilde{\mathcal{A}}_{1}(0,e)$ can be directly decomposed into the sum
of two operators which are similar to $\mathcal{A}(\frac{1}{2},\frac{3}{2},e)$, then
$\dim\emph{Ker}(\widetilde{\mathcal{A}}_{1}(0,e))=6$. For each fixed $e\in[0,1)$, and every eigenvalue
$\lambda_{0,e}=0$ of $\widetilde{\mathcal{A}}_{1}(0,e)$, we assume $\widetilde{x}_{e}=\widetilde{x}_{e}(\theta)$
with unit norm such that $\widetilde{\mathcal{A}}_{1}(0,e)\widetilde{x}_{e}=0$. Then
$\widetilde{\mathcal{A}}_{1}(\eta,e)$ is an analytic path of self-adjoint operators in $\eta$. Following
Kato (\cite{Ko},p.120 and p.386), we can choose a smooth path of unit norm eigenvectors $x_{\eta,e}$ with
$x_{0,e}=\widetilde{x}_{e}$ belonging to a smooth path of real eigenvalues $\lambda_{\eta,e}$ of the
self-adjoint operator $\widetilde{\mathcal{A}}_{1}(\eta,e)$ such that for small enough $\eta$, we have
\bea
\widetilde{\mathcal{A}}_{1}(\eta,e)x_{\eta,e}=\lambda_{\eta,e}x_{\eta,e},\nonumber
\eea
where $\lambda_{0,e}=0$. Then we have
\bea
\frac{\partial}{\partial\eta}\lambda_{\eta,e}|_{\eta=0} =
\<\frac{\partial}{\partial\eta}\widetilde{\mathcal{A}}_{1}(\eta,e)x_{\eta,e},x_{\eta,e}\>|_{\eta=0}. \nonumber
\eea

Let $x_{0,e}=(a, b, c, d)^T$, direct computations show that
\bea
\frac{\partial}{\partial\eta}\lambda_{\eta,e}|_{\eta=0}=
(\alpha^\prime(0,n)+\gamma^\prime(0,n)-\beta^\prime(0,n))\int_{0}^{2\pi}a^2 r_e(\theta) d\theta\nonumber\\
+(\alpha^\prime(0,n)+\gamma^\prime(0,n)+\beta^\prime(0,n))\int_{0}^{2\pi}b^2r_e(\theta)d\theta\nonumber\\
+(\alpha^\prime(0,n)+\gamma^\prime(0,n)+\beta^\prime(0,n))\int_{0}^{2\pi}c^2r_e(\theta)d\theta\nonumber\\
+(\alpha^\prime(0,n)+\gamma^\prime(0,n)-\beta^\prime(0,n))\int_{0}^{2\pi}d^2r_e(\theta)d\nonumber\\
+2\gamma^\prime(0,n)\(\int_{0}^{2\pi}(ac+bd)r_e(\theta)d\theta\). \label{4.91}
\eea
For the case $n=8$, we have
\bea &\alpha^\prime(0,n)\approx1.693575, \ \ \beta^\prime(0,n))\approx 1.792725,\ \
\gamma^\prime(0,n)\approx 0.452, \nonumber \\
&\alpha^\prime(0,n)+\gamma^\prime(0,n)-\beta^\prime(0,n)\approx0.3529, \nonumber\\
&\alpha^\prime(0,n)+\gamma^\prime(0,n)+\beta^\prime(0,n)\approx3.9383, \nonumber\\
&(\alpha^\prime(0,n)+\gamma^\prime(0,n))^2-\beta^\prime(0,n)^2\approx1.3896. \nonumber
\eea
Together with the average value inequality, we obtain
\bea
&&(\alpha^\prime(0,n)+\gamma^\prime(0,n)-\beta^\prime(0,n))\int_{0}^{2\pi} a^2 r_e(\theta)d\theta
   +(\alpha^\prime(0,n)+\gamma^\prime(0,n)+\beta^\prime(0,n))\int_{0}^{2\pi} c^2r_e(\theta) d\theta\nonumber\\
&&\geq2\((\alpha^\prime(0,n)+\gamma^\prime(0,n))^2-\beta^\prime(0,n)^2\)^{\frac{1}{2}}\int_{0}^{2\pi} |ac| r_e(\theta)d\theta \nonumber\\
&&\geq2\gamma^\prime(0,n)\int_{0}^{2\pi} ac r_e(\theta) d\theta \label{4.92}
\eea
If we choose $b$ and $d$ instead of $a$ and $c$, this inequality also holds. Note that the last equality in
(\ref{4.92}) holds only when $ac\equiv 0$. But in this case it easy to check that
$\frac{\partial}{\partial\eta}\lambda_{\eta,e}|_{\eta=0}>0$. Hence we always have
\bea
\frac{\partial}{\partial\eta}\lambda_{\eta,e}|_{\eta=0}>0. \nonumber
\eea
This inequality implies that for any fixed $e\in[0,1)$, there exists a function $\eta_{1}(e)>0$
small enough such that $\widetilde{\mathcal{A}}_{1}(\eta,e)> 0$ for every $\eta\in(0,\eta_{1}(e))$.
Now letting $m_{1}(e)=\frac{1}{\eta_{1}(e)}$, the proof of Lemma \ref{lem4.13} is complete.
\end{proof}

From Lemmas (\ref{lem4.11}), (\ref{lem4.12}) and (\ref{lem4.13}), we have

\begin{lem}\label{thm4.10}
For $n=8$, there exists a function $m_{1}(e)>0$ in $e$ such that
$\mathcal{A}(R_{1},e)>0$ in $\bar{D}_{2}(1,2\pi)$ for all $(m,e)\in (m_{1}(e),+\infty)\times[0,1)$.

For all $n\geq 9$ and $(m,e)\in (0,+\infty)\times[0,1)$, we have $\mathcal{A}(R_{1},e)>0$ in $\bar{D}_{2}(1,2\pi)$.
\end{lem}

\textbf{2)}\ \ For $2\leq l \leq [\frac{n-1}{2}]$, from (\ref{3.16}), (\ref{3.17}), we have
\bea
\mathcal{A}(R_l,e))\geq
\mathcal{A}(\check{R}_l,e),\ \
\mathcal{A}(\check{R}_l,e)=\mathcal{A}(\check{R}_{l,0},e)\oplus\mathcal{A}(\check{R}_{l,0},e),\nonumber
\eea
where
\bea
\mathcal{A}(\check{R}_{l,0},e)&=&-\frac{d^2}{d\theta^2}I_{2}-2J_{2}\frac{d}{d\theta}
       +r_e(\theta)(I_{2}+\frac{1}{2\lambda}(E_{l}+G_{l}-2\tilde{F}_{l})
+\frac{1}{2\lambda}(E_{l}-G_{l})\nonumber\\
&=&-\frac{d^2}{d\theta^2}I_{2}-2J_{2}\frac{d}{d\theta}+r_e(\theta)(I_{2}+\frac{1}{2\lambda}(a_{l}+b_{l}-2S_{l})I_{2}
+\frac{1}{2\lambda}(a_{l}-b_{l})\mathcal{N}), \nonumber \\
\label{4.26}
\eea
with $a_{l}=P_{l}-3Q_{l}+2m$, $b_{l}=P_{l}+3Q_{l}-m$,
$P_{l}=\sum_{j=1}^{n-1}\frac{1-\cos\theta_{jl}\cos\theta_{j}}{2d_{nj}^3}$,
$S_{l}=\sum_{j=1}^{n-1}\frac{\sin\theta_{jl}\sin\theta_{j}}{2d_{nj}^3}$,
$Q_{l}=\sum_{j=1}^{n-1}\frac{\cos \theta_{j}-\cos \theta_{jl}}{2d_{nj}^3}$.
By Corollary \ref{cor4.9}, if
\bea
\frac{1}{3}<\frac{a_{l}+b_{l}-2S_{l}}{a_{l}-b_{l}},\ \ \frac{a_{l}-b_{l}}{2\lambda}<\frac{3}{2},\ \  b_{l}\leq a_{l}, \label{4.89}
\eea
then
\bea
\mathcal{A}(\check{R}_{l,0},e)>0\ \ \mathrm{in}\ \ \bar{D}_{1}(1,2\pi),\quad \mathrm{for}\ \ \forall e\in[0,1). \nonumber
\eea
Hence
\bea
\mathcal{A}(R_l,e)>0\ \ \mathrm{in}\ \ \bar{D}_{2}(1,2\pi),\quad \mathrm{for}\ \ \forall e\in[0,1).\nonumber
\eea

Inequalities in (\ref{4.89}) are equivalent to
\bea
\frac{1}{3}<\frac{2(P_{l}-S_{l})+m}{3m-6Q_{l}},\ \ \frac{3m-6Q_{l}}{2m+\sigma_{n}}<\frac{3}{2}, 2Q_{l}\leq m.\nonumber
\eea
Assume $m>2Q_{\max}(n-1)=2\max\{Q_{l}|2\leq l \leq [\frac{n-1}{2}]\}$ holds. Then from above
inequalities, we need
\bea
-Q_{l}<P_{l}-S_{l},\ \ -6Q_{l}<\frac{3}{2}\sigma_{n}.\nonumber
\eea
From Roberts \cite{Rob1}, we have $P_{l}\geq S_{l}$, $Q_{l}>0$ and $\sigma_{n}>0$. Hence the inequalities
are always true. Now we have

\begin{lem}\label{thm4.11} For $2\leq l \leq [\frac{n-1}{2}]$, we have
\bea
\mathcal{A}(R_l,e)>0\ \ \mathrm{in}\ \ \bar{D}_{2}(1,2\pi),\quad \mathrm{for}\ \
\forall(m,e)\in(2Q_{\max}(n-1),+\infty)\times[0,1),\nonumber
\eea
where $Q_{\max}(n-1)=\max\{Q_{l}|\ \ 2\leq l \leq [\frac{n-1}{2}]\}.$
\end{lem}

\textbf{3)}\ \ For $n\in2\mathbb{N}$, $l=[\frac{n}{2}]$, we have
\bea
\mathcal{A}(R_l,e)=-\frac{d^2}{d\theta^2}I_{2}-2J_{2}\frac{d}{d\theta}+r_e(\theta)(I_{2}+\frac{1}{2\lambda}(a_{l}+b_{l})I_{2}
+\frac{1}{2\lambda}(a_{l}-b_{l})\mathcal{N} ), \label{4.28}
\eea
with $a_{l}=P_{l}-3Q_{l}+2m$ and $b_{l}=P_{l}+3Q_{l}-m$. By Corollary \ref{cor4.9}, if
\bea
\frac{1}{3}<\frac{a_{l}+b_{l}}{a_{l}-b_{l}},\ \ \frac{a_{l}-b_{l}}{2\lambda}<\frac{3}{2},\ \ b_{l}\leq a_{l}, \label{4.86}
\eea
then
\bea
\mathcal{A}(R_l,e)>0\ \ \mathrm{in}\ \ \bar{D}_{1}(1,2\pi),\quad \mathrm{for}\ \ \forall e\in[0,1).\nonumber
\eea
Inequalities in (\ref{4.86}) are equivalent to
\bea
\frac{1}{3}<\frac{2P_{l}+m}{3m-6Q_{l}},\ \ \frac{3m-6Q_{l}}{2m+\sigma_{n}}<\frac{3}{2}, 2Q_{l}\leq m. \nonumber
\eea
Assume $m>2Q_{l}$. Then from above inequalities, we need
\bea
-Q_{l}\leq P_{l},\ \ -6Q_{l}<\frac{3}{2}\sigma_{n}.\nonumber
\eea
Also from Roberts \cite{Rob1}, we have $P_{l}\geq0$, $Q_{l}>0$ and $\sigma_{n}>0$. Hence the
inequalities hold always. Now we have

\begin{lem}\label{thm4.12} For $n\in 2\mathbb{N}$ and $l=[\frac{n}{2}]$, we have
\bea
\mathcal{A}(R_l,e)>0\ \ \mathrm{in}\ \ \bar{D}_{1}(1,2\pi),\quad \mathrm{for}\ \ \forall(m,e)\in(2Q_{l},+\infty)\times[0,1).\nonumber
\eea
\end{lem}

Now by using (\ref{4.8}) and Lemmas \ref{thm4.10}, \ref{thm4.11} and \ref{thm4.12}, we prove that
the following theorem holds for the $(1+n)$-system.

\begin{thm}\label{thm4.13}
If $n\geq 9$, then
\bea
\mathcal{A}(R,e)>0\ \ \mathrm{in}\ \ \bar{D}_{n-1}(1,2\pi),\quad \mathrm{for}\ \ \forall(m,e)\in(2Q_{\max}(n),+\infty)\times[0,1).\nonumber
\eea
If $n=8$, then
\bea
\mathcal{A}(R,e)>0\ \ \mathrm{in}\ \ \bar{D}_{n-1}(1,2\pi),\quad \mathrm{for}\ \ \forall(m,e)\in(\max\{2Q_{\max}(n), m_{1}(e)\},+\infty)\times[0,1).\nonumber
\eea
where $Q_{\max}(n)=\max\{Q_{l}| 2\leq l \leq [\frac{n}{2}]\}.$
\end{thm}

\subsection{Estimate $-1$-Maslov type index of (1+n)-gon system}\label{sec4.6}

\textbf{1)}\ \ For $l=1$, we have
\bea
\mathcal{A}(R_1,e)=-\frac{d^2}{d\theta^2}I_{2}-2\mathbb{J}_{2}\frac{d}{d\theta}+r_e(\theta)R_{1},\nonumber
\eea
since
\bea
\lim_{m\rightarrow+\infty}\frac{1}{\lambda}(\hat{d}_{n}+\frac{m}{2})=\frac{1}{2},\ \
\lim_{m\rightarrow+\infty}\frac{3\sqrt{m(m+n)}}{2\lambda}=\frac{3}{2}.\nonumber
\eea
then from (\ref{2.13}), we have
\bea
\lim_{m\rightarrow+\infty}R_{1}=I_{4}+\left(\begin{array}{cccc} \frac{1}{2}& \frac{3}{2}\\ \frac{3}{2} & \frac{1}{2}\end{array}\right)\diamond \left(\begin{array}{cccc} \frac{1}{2}& -\frac{3}{2}\\ -\frac{3}{2} & \frac{1}{2}\end{array}\right),\nonumber
\eea
and
\bea
\lim_{m\rightarrow+\infty}T^{t}\mathcal{A}(R_1,e)T=\mathcal{A}(\frac{1}{2},\frac{3}{2},e)\oplus\mathcal{A}(\frac{1}{2},-\frac{3}{2},e),\nonumber
\eea
$\mathcal{A}(\frac{1}{2},-\frac{3}{2},e)$ is similar to $\mathcal{A}(\frac{1}{2},\frac{3}{2},e)$. From \eqref{3.27}, we have
\bea
\phi_{-1}(\mathcal{A}(\frac{1}{2},\frac{3}{2},e))=2,\quad \nu_{-1}(\mathcal{A}(\frac{1}{2},\frac{3}{2},e))=0,\quad \forall e\in[0,1).\nonumber
\eea
Hence we have

\begin{lem}\label{thm4.14}
There exists $m^{*}_{1}(e)>0$ depending on $e$ such that
\bea
\phi_{-1}(\mathcal{A}(R_1,e))=4,\quad \nu_{-1}(\mathcal{A}(R_1,e))=0, \qquad \forall m\in[m^{*}_{1}(e),+\infty), e\in[0,1).\nonumber
\eea
\end{lem}
Next, we consider the case $2\leq l \leq [\frac{n-1}{2}]$.
\\
\textbf{2)}\ \ For $2\leq l \leq [\frac{n-1}{2}]$, we have
\bea
\mathcal{A}(R_l,e)=-\frac{d^2}{d\theta^2}I_{2}-2\mathbb{J}_{2}\frac{d}{d\theta}+r_e(\theta)R_{l},\nonumber
\eea
since
\bea
\lim_{m\rightarrow\infty}\frac{a_{l}}{\lambda}=2,\ \
\lim_{m\rightarrow\infty}\frac{b_{l}}{\lambda}=-1,\ \
\lim_{m\rightarrow\infty}\frac{R_{l}}{\lambda}=0.\nonumber
\eea
Then from (\ref{2.13}), we have
\bea
\lim_{m\rightarrow+\infty}R_{l}=I_{4}+2I_2\diamond -I_2,\nonumber
\eea
and
\bea
\lim_{m\rightarrow+\infty}\mathcal{A}(R_{l},e)=\mathcal{A}(\frac{1}{2},\frac{3}{2},e)\oplus\mathcal{A}(\frac{1}{2},\frac{3}{2},e).\nonumber
\eea
Hence, we get the lemma below.

\begin{lem}\label{thm4.15}
For $2\leq l \leq [\frac{n-1}{2}]$, there exists $m^{*}_{l}(e)>0$ depending on $e$ such that
\bea
\phi_{-1}(\mathcal{A}(R_{l},e))=4,\quad \nu_{-1}(\mathcal{A}(R_{l},e))=0, \forall m\in[m^{*}_{l}(e),+\infty), e\in[0,1).\nonumber
\eea
\end{lem}

At last, we consider the case $n\in 2\mathbb{N}$ and $l=[\frac{n}{2}]$.
\\
\textbf{3)}\ \ For $n\in2\mathbb{N}$ and $l=[\frac{n}{2}]$, we have
\bea
\mathcal{A}(R_l,e) = -\frac{d^2}{d\theta^2}I_{2}-2J_{2}\frac{d}{d\theta}
   +r_e(\theta)(I_{2}+\frac{1}{2\lambda}(a_{l}+b_{l})I_{2}+\frac{1}{2\lambda}(a_{l}-b_{l})\mathcal{N}).\nonumber\label{4.28}
\eea
Simple computations show that
\bea
\lim_{\mu\rightarrow+\infty}\frac{1}{2\lambda}(a_{l}+b_{l})=\frac{1}{2},\ \
\lim_{\mu\rightarrow+\infty}\frac{1}{2\lambda}(a_{l}-b_{l})=\frac{3}{2}.\nonumber
\eea
Hence, we have
\bea
\lim_{\mu\rightarrow+\infty}\mathcal{A}(R_l,e)=\mathcal{A}(\frac{1}{2},\frac{3}{2},e),\nonumber
\eea
which corresponds to the Kepler case.  
Hence we have

\begin{lem}\label{thm4.16}
For $n\in 2\mathbb{N}$ and $l=[\frac{n}{2}]$, there exists $m^{*}_{l}(e)>0$ depending on $e$ such that
\bea
\phi_{-1}(\mathcal{A}(R_l,e))=2,\quad \nu_{-1}(\mathcal{A}(R_1,e))=0,\quad \forall m\in[m^{*}_{l}(e),+\infty),\quad  e\in[0,1).\nonumber
\eea
\end{lem}

Now by using (\ref{4.8}) and Lemmas \ref{thm4.14}, \ref{thm4.15} and \ref{thm4.16}, for the $(1+n)$-
system, we have obtained the following theorem.

\begin{thm}\label{thm4.17}
There exits $m^{*}_{\max}(e,n)=\max\{m^{*}_{l}(e)|1\leq l\leq [\frac{n}{2}]\}$ such that
\bea
\phi_{-1}(\mathcal{A}(R,e))= 2n-2,\quad \nu_{-1}(\mathcal{A}(R,e))=0 \quad \forall m\in[m^{*}_{\max}(e,n),+\infty),\quad e\in[0,1).\nonumber
\eea
\end{thm}

Now from the above results, we can give the proof of Theorem \ref{th1.1}.
\\

{\bf Proof of Theorem \ref{th1.1}}
\begin{proof} Note that for $l=2,\cdots,[n/2]$, the path $\ga_l$ is spectrally stable for $m$ large enough by
Theorem \ref{4.1.1}. For $l=1$, the path $\gamma_1$ is spectrally stable for $m$ large enough when $n\geq 8$ from Theorems \ref{thm4.11},
\ref{thm4.12}, \ref{thm4.13} and \ref{thm4.17}. Moreover

i)\ \ For $1\leq l\leq [\frac{n-1}{2}]$, from Lemma $\ref{thm4.10}$, $\ref{thm4.14}$, $\ref{thm4.11}$ and $\ref{thm4.15}$
for $\forall e\in[0,1)$, we have
\bea
i_{1}(\ga_{l})=0,\ \ \nu_{1}(\ga_{l})=0,\ \ i_{-1}(\ga_{l})=4,\ \ \nu_{-1}(\ga_{l})=0,\ \ \mathrm{for}\ \
 m\ \ \mathrm{large\ \ enough}.\nonumber
\eea
From Long's book \cite{Lon4}, the normal form of $\ga_{l}(2\pi)$ may have three possible cases:
$\ga_{l}(2\pi)\approx R(\alpha_{l})\diamond R(\beta_{l})\diamond R(\theta_{l})\diamond R(\phi_{l})$;
$\ga_{l}(2\pi)\approx R(\alpha_{l})\diamond R(\beta_{l})\diamond N_{2}(e^{\sqrt{-1}\theta_{l}}, u_{l})$;
$\ga_{l}(2\pi)\approx N_{2}(e^{\sqrt{-1}\alpha_{l}}, u_{l})\diamond N_{2}(e^{\sqrt{-1}\beta_{l}}, v_{l})$;
also from Page $207$ of Long's book, we have the iteration formula
\bea
i_{-1}(\ga_{l})=i_{1}(\ga_{l})+S_{\ga_{l}(2\pi)}^{+}(1)+
\sum_{0<\theta<\pi}(S_{\ga_{l}(2\pi)}^{+}(e^{\sqrt{-1}\theta})-S_{\ga_{l}(2\pi)}^{-}(e^{\sqrt{-1}\theta}))\nonumber
-S_{\ga_{l}(2\pi)}^{-}(-1),
\eea
hence we have
\bea
4=\sum_{0<\theta<\pi}(S_{\ga_{l}(2\pi)}^{+}(e^{\sqrt{-1}\theta})-S_{\ga_{l}(2\pi)}^{-}(e^{\sqrt{-1}\theta})).\nonumber
\eea
Combining it with the Splitting number of the normal form in \cite{Lon4} Page 198, it's easy to know
that the only possible case is $R(\alpha_{l})\diamond R(\beta_{l})\diamond R(\theta_{l})\diamond R(\phi_{l})$
for some $\alpha_{l}, \beta_{l}, \theta_{l}, \phi_{l}\in(\pi,2\pi)$.

ii)\ \ For $n\in2\mathbb{N}, l=[\frac{n}{2}]$, from lemma \ref{thm4.12}, \ref{thm4.16} for $\forall e\in[0,1)$, we have
\bea
i_{1}(\ga_{l})=0,\ \ \nu_{1}(\ga_{l})=0,\ \ i_{-1}(\ga_{l})=2,\ \ \nu_{-1}(\ga_{l})=0,\ \ \mathrm{for}\ \
 m\ \ \mathrm{large\ \ enough}.\nonumber
\eea
Hence the normal form of $\ga_{l}(2\pi)$ may have two possible cases:
$\ga_{l}(2\pi)\approx R(\alpha_{l})\diamond R(\beta_{l})$;
$\ga_{l}(2\pi)\approx N_{2}(e^{\sqrt{-1}\theta_{l}}, u_{l})$.
Similar to the analysis of i), by using the Splitting number and the iteration formula, we get that the only possible
case of $\ga_{l}(2\pi)$ is $R(\alpha_{l})\diamond R(\beta_{l})$ for some $\alpha_{l}, \beta_{l}\in (\pi,2\pi)$.
\end{proof}



\section{Instability}

From above sections, we have proved that the ERE of the $(1+n)$-gon system is stable if the central mass
$m$ is large enough when $n\ge 8$. In this section, we study the instability of ERE of this system with
a small central mass $m$ for all $n\ge 3$.

From (\ref{4.76}), we know
\bea
\mathcal{A}(R,e)=\mathcal{A}(R_1,e)\oplus\mathcal{A}(R_2,e)\oplus\cdots\oplus\mathcal{A}(R_{[\frac{n}{2}]},e).\nonumber
\eea
Then we have the following theorem.

\begin{thm}\label{thm5.1}
For any $e\in [0,1)$ and $\omega\in\mathbb{U}$, the following holds.

i)\ \ For $l=1$,
\bea
\mathcal{A}(R_1,e)>0,\ \ \mathrm{in}\ \ \bar{D}_2({\omega},2\pi),\ \ \mathrm{for} \ \ \forall m\in[0,\frac{P_{1}}{2}),\nonumber
\eea

ii)\ \ For $2\leq l\leq [\frac{n-1}{2}]$,
\bea
\mathcal{A}(R_l,e)>0,\ \ \mathrm{in}\ \ \bar{D}_2({\omega},2\pi),\ \ \mathrm{for} \ \
\forall m\in(\max\{0,\zeta_{l}\}, \xi_{l}),\nonumber
\eea
where $\zeta_{l}=\min\{\frac{3Q_{l}+S_{l}-P_{l}}{2}$,
$\frac{6Q_{l}-\beta_{0}\min\{\sigma_{n},4(P_{l}-S_{l})\}}{3+2\beta_{0}}\}$, $\xi_{l}=\max\{3Q_{l}+P_{l}-S_{l}$,
$\frac{6Q_{l}+\beta_{0}\min\{\sigma_{n},4(P_{l}-S_{l})\}}{3-2\beta_{0}}\}$, and $\beta_{0}=0.7237$.
\\

iii)\ \ For $n\in2\mathbb{N}$ and $l=[\frac{n}{2}]$,
\bea
\mathcal{A}(R_l,e)>0,\ \ \mathrm{in}\ \ \bar{D}_2({\omega},2\pi),\ \ \ \
        \forall\; m\in(\max\{0,\zeta_{l}\}, \xi_{l}),\nonumber
\eea
where $\zeta_{l}=\min\{\frac{3Q_{l}-P_{l}}{2}$, $\frac{6Q_{l}-\beta_{0}\min\{\sigma_{n},4P_{l}\}}{3+2\beta_{0}}\}$,
$\xi_{l}=\max\{3Q_{l}+P_{l}$, $\frac{6Q_{l}+\beta_{0}\min\{\sigma_{n},4P_{l}\}}{3-2\beta_{0}}\}$, and
$\beta_{0}=0.7237$.
\end{thm}

\begin{proof}
{\it i)}\ \ For $l=1$,
\bea
\mathcal{A}(R_{1},e)=-\frac{d^2}{d\theta^2}I_{4}-2\mathbb{J}_{2}\frac{d}{d\theta}+r_e R_{1},
   \quad R_1=I_{4}+\frac{1}{\lambda}\mathcal{U}(1). \nonumber
\eea
Here $\mathcal{U}(1)$ is given by (\ref{2.17}), which is a $4\times4$ matrix. Direct computations show
that all the eigenvalues of $\mathcal{U}(1)$ are positive when $\mu\in[0,\frac{P_{1}}{2})$. Together with
Lemma \ref{lem4.3}, it yields $\mathcal{A}(R_1,e)>0$ in $\bar{D}_2({\omega},2\pi)$ for
all $\omega\in\mathbb{U}$.

{\it ii)}\ \ For $2\leq l\leq [\frac{n-1}{2}]$, from (\ref{3.16}), (\ref{3.17}), (\ref{3.18}), we have
\bea
&& \mathcal{A}(\check{R}_{l},e)\leq\mathcal{A}(R_{l},e),\ \ \mathrm{in}\ \ \bar{D}_{4}(\omega,2\pi),\nonumber\\
&& \mathcal{A}(\check{R}_{l},e)=\mathcal{A}(\check{R}_{l,0},e)\oplus\mathcal{A}(\check{R}_{l,0},e),\nonumber
\eea
where
\bea
\check{R}_{l,0}(\mu,e)=I_{2}+\frac{1}{2\lambda}(a_{l}+b_{l}-2S_{l})I_{2}
   +\frac{1}{2\lambda}(a_{l}-b_{l})\mathcal{N}. \nonumber
\eea
Direct computations show that $\check{R}_{l,0}(\mu,e)>I_{2}$ when
$m\in(\max\{0, \frac{3Q_{l}+S_{l}-P_{l}}{2}\}, 3Q_{l}+P_{l}-S_{l})$. Together with lemma \ref{lem4.3}, it yields
that when $m\in(\max\{0, \frac{3Q_{l}+S_{l}-P_{l}}{2}\}, 3Q_{l}+P_{l}-S_{l})$, we have
\bea\label{5.1}
\mathcal{A}(R_l,e)>0\ \ \mathrm{in}\ \ \bar{D}_2({\omega},2\pi)\ \ \ \
   \forall \omega\in\mathbb{U}. \eea
Moreover, from Corollary \ref{cor4.10}, we get
\bea
\mathcal{A}(\alpha,\beta,e)>0\ \ \mathrm{in}\ \ \bar{D}_{1}(\omega,2\pi)\ \ \ \ \forall \omega\in\mathbb{U}, \nonumber
\eea
for all $(\beta_{0},e)\in [0,0.7237]\times[0,1)$, $\beta\in [0,\beta_{0})$, and
$\alpha\in(\frac{\beta}{2\beta_{0}},+\infty)$.

Let $\beta_{0}=0.7237$ and $\alpha=\frac{1}{2\lambda}(a_{l}+b_{l}-2S_{l})$. When $a_{l}\geq b_{l}$, let
$\beta=\frac{1}{2\lambda}(a_{l}-b_{l})$, and when $a_{l}\leq b_{l}$, let $\beta=\frac{1}{2\lambda}(b_{l}-a_{l})$.
Then the condition $\beta\in [0, \beta_{0})$ and $\alpha\in(\frac{\beta}{2\beta_{0}},+\infty)$ is equivalent to
\bea\label{5.2}
m\in \left(\max\{0, \frac{6Q_{l}-\beta_{0}
   \min\{\sigma_{n},4(P_{l}-S_{l})\}}{3+2\beta_{0}}\}, \frac{6Q_{l}+\beta_{0}\min\{\sigma_{n},4(P_{l}-S_{l})\}}{3-2\beta_{0}}\right).
\eea
From (\ref{5.1}),(\ref{5.2}), we get
\bea
\mathcal{A}(R_l,e)>0,\ \ \mathrm{in}\ \ \bar{D}_2({\omega},2\pi),\ \ \ \ \forall m\in(\max\{0,\zeta_{l}\}, \xi_{l}),\nonumber
\eea
where $\zeta_{l}=\min\{\frac{3Q_{l}+S_{l}-P_{l}}{2},\frac{6Q_{l}-\beta_{0}\min\{\sigma_{n},4(P_{l}-S_{l})\}}{3+2\beta_{0}}\}$,
$\xi_{l}=\max\{3Q_{l}+P_{l}-S_{l},\frac{6Q_{l}+\beta_{0}\min\{\sigma_{n},4(P_{l}-S_{l})\}}{3-2\beta_{0}}\}$, and
$\beta_{0}=0.7237$.

{\it iii)}\ \ For $n\in 2\mathbb{N}$ and $l=[\frac{n}{2}]$, the proof is similar to that of case ii), and thus is omitted here.
\end{proof}

Let $\Gamma^-_1=0$, $\Gamma^+_1=P_1/2$, $\Gamma^-_l=\max\{0,\zeta_{l}\}$ and $\Gamma^+_l=\xi_l$ for
$l=2,\ldots, [\frac{n}{2}]$.

\begin{cor}\label{cor4.11}
For $n\geq 3$, the ERE of the $(1+n)$-gon system is unstable for all $e\in [0,1)$, if
\bea
m\in[\Gamma^-_1,\Gamma^+_1)\ \ or \ \
m\in(\Gamma^-_l, \Gamma^+_l),\ \ \mathrm{for\;some\;} \ 2\leq l\leq \[\frac{n}{2}\]. \nonumber
\eea
\end{cor}

\begin{proof}
From Theorem \ref{thm5.1}, it is easy to see that when $m\in [\Gamma^-_1, \Gamma^+_1)$ or
$m\in(\Gamma^-_l, \Gamma^+_l)$ for some integer $l\in [2,[\frac{n}{2}]]$,
at least one of the operators $\mathcal{A}(R_l,e), 1\leq l\leq [\frac{n}{2}]$ is positive definite
in $\bar{D}_{2}(\omega,2\pi)$ for all $\omega\in\mathbb{U}$. By Theorem \ref{4.1.1} this implies
that the monodromy matrix $\gamma_1(2\pi)$ or $\gamma_{l}(2\pi)$ is hyperbolic. Hence
$\gamma(2\pi)=\gamma_{1}(2\pi)\diamond\gamma_{2}(2\pi)\diamond\cdots\diamond\gamma_{[\frac{n}{2}]}(2\pi)$
is unstable.
\end{proof}

Then Theorem \ref{th1.2} follows from Theorem \ref{4.1.1} and Corollary \ref{cor4.11}.

Now by direct computations, for n=3, 4, 5, 6, 7, 8, we give the region of the mass parameter
$m$ such that the ERE of the $(1+n)$-gon system is unstable.
\begin{center}
\begin{tabular}{|c|c|c|c|c|}
  \hline
   $(1+n)$-gon: & $\mathcal{A}(R_1,e)>0:$ & $\mathcal{A}(R_2,e)>0:$ & $\mathcal{A}(R_3,e)>0:$ & $\mathcal{A}(R_4,e)>0:$  \\
   $n=3$ & [0, 0.0722) & $\backslash$ & $\backslash$ & $\backslash$ \\
   $n=4$ & [0, 0.1768) & [0, 1.7755) & $\backslash$ & $\backslash$ \\
   $n=5$ & [0, 0.3035) & (0.2613, 3.3148) & $\backslash$ & $\backslash$ \\
   $n=6$ & [0, 0.4472) & (0.5858, 5.0850) & (1.0395, 6.3847) & $\backslash$ \\
   $n=7$ & [0, 0.6047) & (0.9586, 7.0430) & (1.8208, 9.9554) & $\backslash$ \\
   $n=8$ & [0, 0.7740) & (1.3720, 9.1598) & (2.8472, 13.9383) & (2.8969, 15.6593) \\
  \hline
\end{tabular}
\end{center}

Moreover, from the table above, we have much stronger results for $n=3$, $4$, and $5$ as listed
below Theorem \ref{th1.2}.

\medskip


\end{document}